\documentclass[a4paper,oneside, reqno]{amsart}
\usepackage{mathrsfs}
\usepackage{amsfonts}
\usepackage{amssymb}
\usepackage{amsxtra}
\usepackage{dsfont}
\usepackage{hyperref}
\usepackage[english,polish]{babel}

\newcommand{\pl}[1]{\foreignlanguage{polish}{#1}}

\usepackage{color}

\usepackage{scalerel,stackengine}
\stackMath
\newcommand\reallywidehat[1]{%
\savestack{\tmpbox}{\stretchto{%
  \scaleto{%
    \scalerel*[\widthof{\ensuremath{#1}}]{\kern-.5pt\bigwedge\kern-.5pt}%
    {\rule[-\textheight/2]{1ex}{\textheight}}
  }{\textheight}%
}{0.9ex}}%
\stackon[.1pt]{#1}{\tmpbox}%
}

\theoremstyle{plain}
\newtheorem{theorem}{Theorem}
\newtheorem{proposition}{Proposition}[section]
\newtheorem{Corollary}[proposition]{Corollary}
\newtheorem{lemma}[proposition]{Lemma}
\theoremstyle{definition}

\newtheorem{remark}{Remark}[section]

\numberwithin{equation}{section}

\newcounter{thm}

\theoremstyle{plain}

\newcommand{\RR}{\mathbb{R}}
\newcommand{\ZZ}{\mathbb{Z}}
\newcommand{\TT}{\mathbb{T}}
\newcommand{\CC}{\mathbb{C}}
\newcommand{\NN}{\mathbb{N}}

\newcommand{\EE}{\mathbb{E}}

\newcommand{\calP}{\mathcal{P}}

\newcommand{\calF}{\mathcal{F}}

\newcommand{\ind}[1]{{\mathds{1}_{{#1}}}}
\newcommand{\dist}{\operatorname{dist}}

\newcommand{\dif}{\mathrm{d}}

\renewcommand{\atop}[2]{\substack{{#1}\\{#2}}}

\newcommand{\sprod}[2] {{#1 \cdot #2}}

\newcommand{\eps}{\epsilon}

\newcommand{\st}{\mathrm{sin}}

\title[Dimension-free estimates in $\ZZ^d$]
{Dimension-free estimates for discrete 
 Hardy--Littlewood averaging operators over the cubes in $\ZZ^d$ }

\author{Jean Bourgain}
\address{Jean Bourgain \\
  School of Mathematics\\
  Institute for Advanced Study\\
  Princeton, NJ 08540\\
  USA}
\email{bourgain@math.ias.edu}

\author{Mariusz Mirek}
\address{Mariusz Mirek \\
  Department of Mathematics\\
  Rutgers University\\
Piscataway, NJ 08854\\ USA \&
	Instytut Matematyczny\\
	Uniwersytet \pl{Wroc{\lll}awski}\\
	Plac Grun\-waldzki 2/4\\
	50-384 \pl{Wroc{\lll}aw}\\
	Poland}
\email{mariusz.mirek@rutgers.edu}

\author{Elias M. Stein}
\address{
	Elias M. Stein\\
	Department of Mathematics\\
	Princeton University\\
	Princeton\\
	NJ 08544-100 USA}
\email{stein@math.princeton.edu}

\author{B{\l}a{\.z}ej Wr{\'o}bel}
\address{ B{\l}a{\.z}ej Wr{\'o}bel\\
	Instytut Matematyczny\\
	Uniwersytet \pl{Wroc{\lll}awski}\\
	Plac Grun\-waldzki 2/4\\
	50-384 \pl{Wroc{\lll}aw}\\
	Poland}
\email{blazej.wrobel@math.uni.wroc.pl}

\thanks{ Jean Bourgain was partially supported by NSF grant DMS-1301619.\  Mariusz
  Mirek was partially  supported by the Schmidt Fellowship and the IAS Found for
  Math.\ and by the National Science Center (Poland), NCN grant  DEC-2015/19/B/ST1/01149.\
  Elias M. Stein was partially supported by NSF grant
  DMS-1265524.  B{\l}a{\.z}ej Wr{\'o}bel was partially supported by
  the National Science Centre (Poland), NCN grant 2014\slash 15\slash D\slash ST1\slash 00405}

\begin{document}
 
\selectlanguage{english}

\begin{abstract}
Dimension-free bounds will be provided in
maximal and $r$-variational inequalities on $\ell^p(\ZZ^d)$ corresponding to the discrete
Hardy--Littlewood averaging operators defined over the cubes in
$\ZZ^d$. We will also construct an example of a  symmetric convex body
in $\ZZ^d$ for which maximal dimension-free bounds fail on $\ell^p(\ZZ^d)$
for all $p\in(1, \infty)$. Finally, some applications in ergodic theory will
be discussed. 
\end{abstract}

\maketitle

\section{Introduction and notation}

\label{sec:1}
In the 1980s dimension-free estimates for the Hardy--Littlewood
maximal functions over symmetric convex bodies had begun to be studied
and gone through a period of considerable changes and
developments. This line of research was originated by the third author
in \cite{SteinMax}, see also
\cite{StStr}, where dimension-free bounds for the
Hardy--Littlewood maximal functions over the Euclidean balls were obtained on $L^p(\RR^d)$ for
$p \in(1, \infty]$.
Averages over general symmetric convex bodies were
considered  in \cite{B1, B2, Car1, Mul1}. We refer also to more
recent results \cite{Ald1, B3,  NaTa} and the survey article
\cite{DGM1} for a very careful and exhaustive exposition of the
subject. However, at that time the discrete analogues of these
dimension-free estimates had not been investigated, and only recently
has the dimension-free role of $r$-variations been broached \cite{BMSW1}.

In this article we initiate systematic studies of the estimates
independent of the dimension for the Hardy--Littlewood averaging
operators in the discrete setup. On the one hand, we give a
counterexample that shows that the phenomenon of dimension-free
estimates in the discrete setting cannot be as broad as in the
continuous setting. On the other hand, for  the discrete
Hardy--Littlewood averaging operators over the cubes in $\ZZ^d$ some
positive results will be proved here.
We will also discuss dimension-free $r$-variational
estimates and their  applications to ergodic theory.

Let $G$ be a bounded, closed and symmetric convex subset of $\RR^d$
with non-empty interior. Throughout the paper such a set $G$ will be
called a symmetric convex body. We remark that usually in the
literature a symmetric convex body $G$ is assumed to be open. In
fact, when averaging operators over convex sets in $\RR^d$ are
considered there is no difference whether we assume $G$ is closed or
open, since the boundary of a convex set has Lebesgue measure zero.
However, in the discrete case in order to avoid some
technicalities, we will assume that a symmetric
convex body $G$ is always closed.

 For every $x\in\ZZ^d$ and $t>0$ and for
every function $f\in\ell^1(\ZZ^d)$ let
\begin{align}
  \label{eq:0}
\mathcal M_t^Gf(x)=\frac{1}{|G_t\cap \ZZ^d|}\sum_{y\in G_t\cap\ZZ^d}f(x-y)
\end{align}  
be the discrete Hardy--Littlewood averaging operator over
$G_t\cap \ZZ^d$, where $G_t=\{y\in\RR^d: t^{-1}y\in G\}$.

The operator $\mathcal M_t^G$ is a convolution operator with the kernel
 \begin{align*}
\mathcal K_t^G(x)=\frac{1}{|G_t\cap \ZZ^d|}\sum_{m\in G_t\cap \ZZ^d}\delta_m(x),   
 \end{align*}
 where $\delta_m$ stands for the Dirac's delta at $m\in\ZZ^d$.  

 It is natural that $\mathcal M_t^G$ can be thought of as a discrete analogue of the integral Hardy--Littlewood averaging operator
 \begin{align}
   \label{eq:112}
M_t^Gf(x)=\frac{1}{|G_t|}\int_{G_t}f(x-y){\rm d}y,
 \end{align}
 defined for every $f\in L^1_{\rm loc}(\RR^d)$.
\subsection{Maximal estimates}
 We know from \cite{
   B2, Car1} that for every $p\in(3/2, \infty]$, there is $C_p>0$
 independent of the dimension such that for every convex symmetric body $G\subset\RR^d$  and for  every $f\in L^p(\RR^d)$ we have
 \begin{align}
   \label{eq:113}
   \big\|\sup_{t>0}|M_t^Gf|\big\|_{L^p}\le C_p\|f\|_{L^p}.
 \end{align}
For the dyadic/lacunary variant of $M_t^G$ the range of $p$'s can be extended and one can show that
 for every $p\in(1, \infty]$, there is $C_p>0$
 independent of the dimension such that for every convex symmetric body $G\subset\RR^d$ and  for every $f\in L^p(\RR^d)$  we have
 \begin{align}
\label{eq:115}
   \big\|\sup_{n\in\ZZ}|M_{2^n}^Gf|\big\|_{L^p}\le C_p\|f\|_{L^p}.
 \end{align}
It is conjectured that the inequality in \eqref{eq:113} holds for all $p\in(1, \infty]$ and for all convex symmetric bodies $G\subset\RR^d$ with the implied constant independent of $d$. It is reasonable to believe that this is true, since it has been verified for a large class of convex symmetric bodies.   
If $G=B^q$ for $q\in[1, \infty]$, where $B^q$ is a  ball induced by a small $\ell^q$
norm in $\RR^d$, i.e. 
\begin{align}
  \label{eq:114}
  \begin{split}
B^q=&\Big\{x=(x_1, \ldots, x_d)\in\RR^d: |x|_q=\Big(\sum_{1\le k\le
  d}|x_k|^q\Big)^{1/q}\le1\Big\}, \\
B^{\infty}&=\big\{x=(x_1, \ldots, x_d)\in\RR^d: |x|_{\infty}=\max_{1\le k\le d}|x_k|\le1\big\},
\end{split}
\end{align}
then the inequality in \eqref{eq:113} holds for all $p\in(1, \infty]$
with a constant $C_{p, q}>0$, which is independent of the
dimension. The case $G=B^q$ for $q\in[1, \infty)$ was handled in
\cite{Mul1} and the case $G=B^{\infty}$ of cubes was recently handled
by the first author in \cite{B3}. In the case of cubes, we remark that the inequality \eqref{eq:113}
for all $p\in(1, \infty]$ cannot be obtained by the interpolation of $L^{\infty}(\RR^d)$ bound from \eqref{eq:113} with  the
weak type $(1,1)$  estimate established in  Aldaz's paper \cite{Ald1}, since the latter bounds involve constants that are unbounded as $d\to\infty$.

\medskip

In the first part of the paper our aim now will be to understand whether it is possible to obtain $\ell^p(\ZZ^d)$ inequalities for the  maximal function associated with $\mathcal M_t^G$ with bounds independent of the dimension. It is not difficult to see (appealing to a covering argument for $p=1$) that for every $p\in(1, \infty]$ and for every symmetric convex body $G\subset\RR^d$ there is a constant $C_{p}(d)>0$ such that for every $f\in\ell^p(\ZZ^d)$ we have
\begin{align}
  \label{eq:43}
   \big\|\sup_{t>0}|\mathcal M_t^Gf|\big\|_{\ell^p}\le C_p(d)\|f\|_{\ell^p}.  
\end{align}
Of course, if $p=\infty$ there is nothing to do, since $\mathcal M_t^G$ is an averaging operator and \eqref{eq:43} holds with bounds independent of the dimension for every $G\subset\RR^d$. Therefore, only parameters $p\in(1, \infty)$ will matter.

At first glance one thinks that it should be possible, in view of
\eqref{eq:113}, to deduce bounds  in \eqref{eq:43} that are independent of $d$ 
on $\ell^p(\ZZ^d)$  from the  dimension-free results on $L^p(\RR^d)$ by
comparison of the maximal function corresponding to $\mathcal M_t^G$
on $\ZZ^d$ with the maximal function corresponding to $M_t^G$ on
$\RR^d$. This idea only gives a partial answer. Namely, we have the following general result.
\begin{theorem}
\label{thm:8}
For a closed symmetric convex body $G\subset \RR^d$ we define the constant
\begin{equation}
\label{eq:cg} c(G):=\inf\{t>0\colon Q_{1/2}\subseteq tG\},
\end{equation}
where $Q_{1/2}=[-1/2,1/2]^d$. Then for every  $p\in(1, \infty)$ the following inequality
	\begin{equation}
	\label{eq:trEu1}
	 \big\|\sup_{t\ge c(G)d}|\mathcal M_t^Gf|\big\|_{\ell^p}\leq e^6 \big\|M_*^G\big\|_{L^p(\RR^d)\to L^p(\RR^d)}\|f\|_{\ell^p}
         \end{equation}
         holds for all $f\in\ell^p(\ZZ^d)$. 
\end{theorem}

This simple comparison argument will allow us to deduce dimension-free
estimates for those discrete maximal functions whose supremum is taken
over $t\ge c(G)d$ as long as the corresponding dimension-free bounds are
available for their continuous analogues. At this stage, the whole
difficulty lies in estimating $\sup_{0<t\le c(G)d}|\mathcal M_t^Gf|$,
and here the things are getting more complicated.

We shall show that the dimension-free estimates in the discrete case
are not as broad as in the continuous setup by constructing an
example of a symmetric convex body in $\ZZ^d$ for which maximal
estimates on $\ell^p(\ZZ^d)$ for every $p\in(1, \infty)$ involve constants
which grow to infinity as $d\to \infty$.

Namely, let
$1\leq \lambda_1<\cdots<\lambda_d<\ldots<\sqrt{2}$ be a fixed sequence and define  the ellipsoid 
\begin{equation}
\label{eq:elip}
E_d:=\Big\{x\in \RR^d\colon \sum_{j=1}^d \lambda_j^2\, x_j^2\,\le 1 \Big\}.
\end{equation}
Then on the one hand, in view of the comparison principle described in
Theorem \ref{thm:8} and inequality \eqref{eq:113} with
$G=E_d\subset\RR^d$, one is able to show that for every
$p\in(3/2, \infty]$ there is $C_p>0$ independent of $d\in\NN$ such
that the following estimate
\begin{align}
\label{eq:49}
   \big\|\sup_{t\ge d^{3/2}}|\mathcal M_t^{E_d}f|\big\|_{\ell^p}\le C_p\|f\|_{\ell^p}
\end{align}
holds for every $f\in\ell^p(\ZZ^d)$, since $\frac{1}{2}d^{1/2}\le c(E_d)\le d^{1/2}$, see Section \ref{sec:2} for more details.

On the other hand, Theorem \ref{thm:9} shows that \eqref{eq:49} is not
true if the full maximal function corresponding to $\mathcal M_t^{E_d}$ is
considered. Namely, denoting \begin{equation}
\label{eq:Cepd}
\mathcal C_p(E_d):=\sup_{\|f\|_{\ell^p}\le1} \big\| \sup_{t>0}|\mathcal M_t^{E_d} f|\big\|_{\ell^p},
\end{equation}   we have the following result.

\begin{theorem}
\label{thm:9}
For every $p\in(1, \infty)$ there exists a constant $C_p>0$ such that for all $d\in \NN$ we have
\[
\mathcal C_p(E_d)\ge C_p\cdot(\log d)^{1/p}.
\]
\end{theorem}
Theorem \ref{thm:9} shows that the question about the dimension-free
estimates in the discrete setting for the Hardy-Littlewood maximal
functions is much more delicate and there is no obvious conjecture to
prove.  So, it is interesting to know whether we can expect bounds
independent of the dimension on $\ell^p(\ZZ^d)$ with $p\in(1, \infty)$
for the discrete maximal function $\sup_{t>0}|\mathcal M_t^{B^q}f|$,
where $B^q$ is a ball as in \eqref{eq:114} with $q\in[1,
\infty)$. This question is considerably harder due to the lack of
reasonable error estimates depending on $d$ for the number of lattice points in the sets
$B^q$ and new methods will surely need to be invented.  Therefore, even
the $\ell^2$ theory is very intriguing. In the the ongoing project
\cite{BMSW2} we initiated investigations in this direction and the
context of the discrete Euclidean balls $B^2$ is studied.

However, if $q=\infty$ then $B^{\infty}_t=[-t, t]^d$ is a cube and an accurate count for the
number of lattice points is not a problem any more. The product
structure of the cubes allows us to count the number of lattice points
in $B^{\infty}_t$ and we get
$|B^{\infty}_t\cap\ZZ^d|=(2\lfloor t\rfloor+1)^d$. This property distinguishes the cubes from the $B^q$ balls for $q\in[1, \infty)$ and in some sense encourages us to think that the inequality \eqref{eq:43} may hold with the bound independent of the dimension for a certain range of $p$'s.

Form now on, for simplicity
of the notation we will write $Q_t=[-t, t]^d$ for $t>0$ and $Q=[-1, 1]^d$. We shall
provide analogues of  inequalities \eqref{eq:113} and
\eqref{eq:115} for the discrete operators $\mathcal M_t^{Q}$ over the cubes
$Q_t\cap\ZZ^d$.
One of the  main theorems of this paper is the following maximal result.
\begin{theorem}
  \label{thm:1}
  For every $p\in(3/2, \infty]$ there exists a constant $C_p>0$ such that for
  every $d\in\NN$ and every $f\in\ell^p(\ZZ^d)$ we have
  \begin{align}
    \label{eq:68}
    \big\|\sup_{t>0}|\mathcal M_t^Qf|\big\|_{\ell^p}\le  C_p\|f\|_{\ell^p}.
  \end{align}
\end{theorem}
 If we restrict the supremum in \eqref{eq:68} to the dyadic times, i.e. $t\in\{2^n: n\in\NN_0\}$, where $\NN_0=\NN\cup\{0\}$,  then the range of $p$'s can be improved.  

 \begin{theorem}
  \label{thm:0}
  For every $p\in(1, \infty]$ there exists a constant $C_p>0$ such that for
  every $d\in\NN$ and every $f\in\ell^p(\ZZ^d)$ we have
  \begin{align}
    \label{eq:87'}
    \big\|\sup_{n\in\NN_0}|\mathcal M_{2^n}^Qf|\big\|_{\ell^p}\le  C_p\|f\|_{\ell^p}.
  \end{align}
\end{theorem}

In fact, in Section \ref{sec:4} we prove a stronger result and we show that  the maximal  inequality from \eqref{eq:87'} holds with any $a_n$ in place of $2^n$, where $(a_n: n\in\NN_0)$ is a lacunary sequence\footnote{A sequence $(a_n: n\in\NN_0)\subseteq(0, \infty)$ is called lacunary, if $\inf_{n\in\NN_0}\frac{a_{n+1}}{a_n}>1$.} in $(0, \infty)$.

In the second part of the paper our aim will be to strengthen the maximal estimates from \eqref{eq:68} and \eqref{eq:87'} and provide $r$-variational estimates independent of the dimension for the operators $\mathcal M_{t}^Q$.

\subsection{$r$-variational estimates}
Very recently, in \cite{BMSW1} we studied estimates independent of the
dimension for the averaging operators \eqref{eq:113} in the context of
$r$-variational seminorms.  Recall that for $r\in[1, \infty)$ the
$r$-variation seminorm $V_r$ of a complex-valued function
$(0, \infty)\times X\ni (t, x)\mapsto\mathfrak a_t(x)$ on some measure
space $(X, \mathcal B(X), \mu)$  is defined by setting
\[
V_r(\mathfrak a_t(x): t\in Z)=\sup_{\atop{t_0<\ldots <t_J}{t_j\in Z}}
\bigg(\sum_{j=0}^J|\mathfrak a_{t_{j+1}}(x)-\mathfrak a_{t_j}(x)|^r\bigg)^{1/r},
\] 
where $Z$ is a subset of $(0, \infty)$ and the supremum is taken over
all finite increasing sequences in $Z$. If $Z$ is the dyadic set $\{2^n:n\in\ZZ\}$ then the $r$-variation $V_r$ is often called the long $r$-variation seminorm.

In what follows we will assume that
$(0, \infty)\ni t\mapsto \mathfrak a_t(x)$ is  a continuous
function for every $x\in X$ or that $Z$ is countable, then there is no problem with 
the measurability of $V_r(\mathfrak a_t(x): t\in Z)$. In the discrete setup the
function $(0, \infty)\ni t\mapsto|G_t\cap\ZZ^d|$ takes only countably many
values, so the parameter $t$ will be always restricted to a countable subset
of $(0, \infty)$. In the case of the discrete cubes $Q_t\cap\ZZ^d$ we will have  $Z=\NN$.

The $r$-variational seminorm is a very useful tool in pointwise
convergence problems. If for some $r\in [1, \infty)$ and $x\in X$ we
have
\[
V_r(\mathfrak a_t(x): t>0)<\infty
\]
then the limits $\lim_{t\to0}a_t(x) $ and $\lim_{t\to\infty}a_t(x)$
exist. So we do not need to establish pointwise convergence on a dense
class as it is usually done in the classical approach. This is very
important while pointwise convergence problems are discussed in the
ergodic context and there is no easy way to find a candidate for such a
dense class.  However, $V_r$ is more difficult to bound than the
maximal function, since it dominates the supremum norm, i.e. for any
$t_0>0$ we have
\[
\sup_{t>0}|\mathfrak a_t(x)|\le |\mathfrak a_{t_0}(x)| + 2V_r(\mathfrak a_t(x): t>0).
\]
There is an extensive literature about the  $r$-variational estimates. For the purposes of this article the most relevant will be
\cite{JR1}, \cite{JSW} and \cite{MTS1}, see also the references given there.

In \cite{BMSW1} we proved that for every $p\in(3/2, 4)$ and for every $r\in(2, \infty)$  there exists a
  constant $C_{p, r}>0$ independent of the dimension $d\in\NN$ such that for every symmetric convex body $G\subset\RR^d$ 
  and for all $f\in L^p(\RR^d)$ we have
  \begin{align}
\label{eq:119}
    \big\|V_r\big(M_{t}^Gf: t>0\big)\big\|_{L^p}\le C_{p, r}\|f\|_{L^p}.
  \end{align}
 The range for the parameter $p$
  in \eqref{eq:119} can be improved if we consider only long
  $r$-variations. Namely, for all
$p\in(1, \infty)$ and $r\in(2, \infty)$  we have
\begin{align}
\label{eq:44}
    \big\|V_r\big(M_{2^n}^Gf: n\in\ZZ\big)\big\|_{L^p}\le C_{p, r}\|f\|_{L^p}.
  \end{align}
Moreover, if $G=B^q$ for $q\in[1, \infty]$ and $B^q$ is a ball as in \eqref{eq:114}, then the inequality \eqref{eq:119} holds for all $p\in(1, \infty)$ and $r\in(2, \infty)$ with a constant $C_{p, q, r}>0$ independent of the dimension. 

The results have been encouraging enough to merit further investigation, especially in the discrete setup. 
Therefore, in the second part of the paper we will be concerned with estimating $r$-variations  for the discrete operators  $\mathcal M_t^Q$ over the cubes with bounds independent of the dimension as in \eqref{eq:119} and \eqref{eq:44}. 
The next theorem is a variational counterpart of Theorem \ref{thm:1}.
\begin{theorem}
\label{thm:10}
  Let $p\in(3/2, 4)$ and $r\in(2, \infty).$  Then there exists a
  constant $C_{p, r}>0$ independent of the dimension $d\in\NN$ and such that for every  $f\in \ell^p(\ZZ^d)$
 we have
  \begin{align}
\label{eq:86}
    \big\|V_r\big(\mathcal M_{t}^Qf: t>0\big)\big\|_{\ell^p}\le C_{p, r}\|f\|_{\ell^p}.
  \end{align}
\end{theorem}
Theorem \ref{thm:11} is a dyadic $r$-variational variant of Theorem \ref{thm:0}, and provides the sharp range of exponents for the parameters $p\in(1, \infty)$ and $r\in(2, \infty)$. 
\begin{theorem}
\label{thm:11}
 Let  $p\in(1, \infty)$ and $r\in(2, \infty)$. Then there exists a
  constant $C_{p, r}>0$ independent of the dimension $d\in\NN$ and such that for every  $f\in \ell^p(\ZZ^d)$
 we have
  \begin{align}
\label{eq:74}
    \big\|V_r\big(\mathcal M_{ 2^n}^Qf: n\in\NN_0\big)\big\|_{\ell^p}\le C_{p, r}\|f\|_{\ell^p}.
  \end{align}
  \end{theorem}
The range for parameter $r\in(2, \infty)$ in Theorem \ref{thm:10} and   \ref{thm:11} is sharp, see for instance  \cite{JSW}. Dimension dependent versions of Theorem \ref{thm:10} and  \ref{thm:11}, with sharp ranges of parameters $p\in(1, \infty)$ and $r\in(2, \infty)$, may be easily proven using the methods of the paper.

  Finally some applications of Theorem \ref{thm:10} and Theorem
  \ref{thm:11} will be discussed.  These $r$-variational results have
  a natural ergodic theoretical interpretation and will be discussed in the next paragraph.

\subsection{Applications}  Let
$(X, \mathcal B(X), \mu)$ be a $\sigma$-finite measure space with
a family of commuting and invertible measure-preserving  transformations
$T_1, T_2, \ldots, T_d$
which map $X$ to itself. For every  $f\in L^1(X)$ we define
the ergodic Hardy--Littlewood averaging operator by setting
\begin{align}
  \label{eq:118}
      \mathcal A_t^Gf(x)=\frac{1}{|G_t\cap \ZZ^d|}\sum_{y\in G_t\cap\ZZ^d}f\big(T_1^{y_1}\circ T_2^{y_2}\circ\ldots\circ T_d^{y_d}x\big).
  \end{align}
The operator $\mathcal A_t^G$ can be thought of as an ergodic counterpart of $\mathcal M_t^G$. Indeed, it suffices to take $X=\ZZ^{d}$,
$\mathcal B(\ZZ^d)$ the $\sigma$-algebra of all subsets of
$\ZZ^{d}$, $\mu=|\cdot|$ to be the counting
measure on $\ZZ^{d}$ and
$S_{j}^y:\ZZ^{d} \rightarrow \ZZ^{d}$ the shift
operator acting of $j$-th coordinate,
i.e. $S_{j}^y(x)=x-ye_{j}$ for all $j\in\{1, \ldots, d\}$ and
$y\in\ZZ$, where $e_{j}$ is the $j$-th basis vector from
the standard basis in $\ZZ^{d}$. 

For the operators $\mathcal A_{N}^Q$ defined over the cubes we also have dimension-free $r$-variational estimates.

\begin{theorem}
\label{thm:12}
 Let $p\in(3/2, 4)$ and $r\in(2, \infty).$  Then there exists a
  constant $C_{p, r}>0$ independent of the dimension $d\in\NN$ such that for all  $f\in L^p(X)$ the following inequality
holds
  \begin{align}
\label{eq:120}
    \big\|V_r\big(\mathcal A_{N}^Qf: N\in\NN\big)\big\|_{L^p}\le C_{p, r}\|f\|_{L^p}.
  \end{align}
 Moreover, if we  consider only long variations, then \eqref{eq:120} remains true for all 
  $p\in(1, \infty)$ and $r\in(2, \infty)$ and we have
 \begin{align}
\label{eq:121}
    \big\|V_r\big(\mathcal A_{2^n}^Qf: n\in\NN_0\big)\big\|_{L^p}\le C_{p, r}\|f\|_{L^p}.
  \end{align} 
\end{theorem}

In Proposition \ref{prop:3} we provide a transference principle, which
allows us to derive inequalities \eqref{eq:120}, \eqref{eq:121} from
the corresponding estimates in \eqref{eq:86} and \eqref{eq:74}
respectively. Now two remarks are in order. Firstly, the remarkable
feature of Theorem \ref{thm:12} is that the implied bounds in
\eqref{eq:120}, \eqref{eq:121} are independent of the number of
underlying  transformations
$T_1,\ldots,T_d$.  Secondly, for the operators $\mathcal A_t^Qf$,
which are defined on an abstract measure space, there is no obvious
way how to find a candidate for a dense class to establish pointwise
convergence.  Fortunately, due to the properties of $r$-variational
seminorm we immediately know that the limit
$\lim_{t\to\infty}\mathcal A_t^Qf(x)$ exists almost everywhere on $X$
for every $f\in L^p(X)$ and the desired conclusion follows directly.

\subsection{Overview of the methods}
We shall briefly outline the strategy for proving our main results.
The first step in the proofs of Theorem \ref{thm:1} and Theorem
\ref{thm:0} will rely to a large extent on an adaptation of Carbery's
almost orthogonality principle \cite[Theorem 2]{Car1}, which is stated
as Proposition \ref{prop:2} in the paper. In the second step we are
reduced to verify the  assumptions of Proposition \ref{prop:2}. In order to
 do this we have to construct a suitable symmetric
diffusion semigroup $P_t$, provide dimension-free estimates for the
multiplier $\mathfrak m_t^{Q}$ corresponding to the operator
$\mathcal M_t^{Q}$ and finally we have to control the maximal function
$\sup_{2^n\le t<2^{n+1}}|\mathcal M_t^{Q}f|$ on $\ell^p(\ZZ^d)$ in a
certain range of $p$'s.  However, due to the discrete nature of our
questions, the methods employed in the continuous setting in
\cite{B3}, \cite{Car1}, and \cite{Mul1} for verifying underlying
assumptions do not easily adapt to the discrete setting.

Fortunately, for the operators $\mathcal M_t^{Q}$ 
over the cubes in $\ZZ^d$ we will be able to obtain the desired conclusions.
We begin by constructing  a suitable symmetric diffusion semigroup $P_t$ introduced in Section \ref{sec:4}. The semigroup $P_t$ in our case corresponds to the discrete Laplacian on $\ZZ^d$, and provides maximal and $r$-variational estimates and the Littlewood--Paley theory with bounds independent of the dimension, which one obtains by appealing to the general   
theory of symmetric diffusion semigroups in the sense of \cite[Chapter III]{Ste1}.

Further, we have to understand the behavior of the multiplier
$\mathfrak m_t^{Q}$ associated with the operator
$\mathcal M_t^{Q}$. This in turn is an exponential sum, which is the
product of one dimensional Dirichlet's kernels. The explicit formula
for $\mathfrak m_t^{Q}$ in terms of the Dirichlet kernels is essential
for the further calculations and allows us to furnish the bounds
independent of the dimension as described in \eqref{eq:4}. The
inequalities in \eqref{eq:4} are based on elementary estimates, which
are interesting in their own right. For this reason our method does
not extend to discrete convex bodies other than $Q$.  This is the
second place which sets the operators $\mathcal M_t^{Q}$ over the
cubes apart from the operators $\mathcal M_t^{B^q}$ for $q\in[1, \infty)$,
where $B^q$ is a ball as in \eqref{eq:114}. The multiplier
$\mathfrak m_t^{B^q}$ associated with the operator
$\mathcal M_t^{B^q}$ is again an exponential sum, however the absence
of the product structure makes the estimates incomparably harder. The
estimates for $\mathfrak m_t^{B^q}$, which are a part of the ongoing
project \cite{BMSW2}, are based on delicate combinatorial arguments,
which differ completely from the methods of estimates for
$\mathfrak m_t^{Q}$ provided in Section \ref{sec:3}.

The crucial new ingredient we shall use is a numerical variant of
Radmeacher--Menshov inequality, as in \cite{BMSW1}, which asserts that
for every $n\in\NN_{0}$ and for every function
$\mathfrak a:[2^n, 2^{n+1}]\cap\NN\to\CC$ and $r\ge1$ we have
        \begin{align}
          \label{eq:42}
          \begin{split}
            	\sup_{2^n\le t< 2^{n+1}}|\mathfrak a(t)-\mathfrak a(2^n)|&\le V_r\big(\mathfrak a(t): t\in[2^n, 2^{n+1})\big)\\
		&\le
		2^{1-1/r}
		\sum_{0\le l\le n}
		\Big(
		\sum_{k = 0}^{2^{l}-1}
		\big|\mathfrak a(2^n+{2^{n-l}(k+1)}) -\mathfrak a(2^n+{2^{n-l}k})\big|^r
		\Big)^{1/r}.
          \end{split}
        \end{align}
        Inequality \eqref{eq:42} replaces the fractional integration
        argument from \cite{Car1} (as it is not clear if this argument is  
        available in the discrete setting) and allows us to obtain
        \eqref{eq:68} for $p\in(3/2, 2]$.  A variant of this
        inequality was proven by Lewko--Lewko \cite[Lemma 13]{LL} in
        the context of variational Rademacher--Menshov type results
        for orthonormal systems and it was also obtained independently
        by the second author and Trojan \cite[Lemma 1]{MT1} in the
        context of variational estimates for discrete Radon
        transforms, see also \cite{MTS1}.  Inequality \eqref{eq:42}
        reduces estimates for a supremum or an $r$-variation
        restricted to a dyadic block to the situation of certain
        square functions, where the division intervals over which
        differences are taken (in these square functions) are all of
        the same size. Inequality \eqref{eq:42}, combined with the
        estimates from \eqref{eq:4}, is an invaluable tool in
        establishing the following maximal bound
        \begin{align}
          \label{eq:50}
              \sup_{n\in\NN_0}\big\|\sup_{2^n\le t< 2^{n+1}}|\mathcal M_t^Qf|\big\|_{\ell^p}\le  C_p\|f\|_{\ell^p}
        \end{align}
for all $f\in\ell^p(\ZZ^d)$ and $p\in(3/2, \infty]$ with some constant $C_p>0$, which does depend on the  dimension.

Gathering now all together and invoking Proposition \ref{prop:2} and
dimension-free Littlewood--Paley inequality from \eqref{eq:98} we may
extend inequality \eqref{eq:50} to the full maximal inequality
\eqref{eq:68} for all $p\in(3/2, \infty]$, with the implied bound
which does not depend on $d\in\NN$. In the dyadic case we do not need
to prove inequality \eqref{eq:50} and this is loosely speaking the
reason why we obtain \eqref{eq:87'} for all $p\in(1, \infty]$.  It is
worth emphasizing that the method described above can be used to
obtain \eqref{eq:113} and \eqref{eq:115} without appealing to the
fractional integration method.

The approach undertaken in this paper is robust enough to provide
$r$-variational dimension-free estimates for the operators
$\mathcal M_t^Q$. We now briefly outline the key steps for proving
Theorem \ref{thm:10} and \ref{thm:11}.

We first split the consideration into long and short variations as in
\eqref{eq:29}.  The long variations \eqref{eq:30} are handled in Theorem
\ref{thm:11} by invoking the dimension-free estimates for
$r$-variations of the semigroup $P_t$. We refer to \cite[Theorem
3.3]{JR1} or \cite[inequality (2.30)]{BMSW1} for more details. To
establish Theorem \ref{thm:11} it remains to control the error term,
which is handled by the square function methods, and the
Littlewood--Paley theory, see \eqref{eq:31}.

The analysis of short variations \eqref{eq:36} breaks into two cases, whether
$p\in[2, 4)$ or $p\in(3/2, 2]$. In the first case for $p\in[2, 4)$ we
use  the square function methods, and the Littlewood--Paley theory and
reduce the estimates  basically to Theorem \ref{thm:0}. In the second
case for $p\in(3/2, 2]$ we proceed actually very much in the spirit of
the proof of Theorem \ref{thm:1}. Namely, we rely on 
inequality \eqref{eq:42} and adapt the methods of the proof of
Proposition \ref{prop:2} to the $r$-variational case, in fact with $r=2$, which 
is suited to an application of the Fourier transform techniques with estimates from \eqref{eq:4}.

There is a natural question which now arises. Is it possible to extend
the range of $p$'s in Theorem \ref{thm:1} to $p\in(1,3/2]$? For the
maximal function associated with the operators $M_t^Q$ over the cubes in
$\RR^d$ given by \eqref{eq:112} this was accomplished in
\cite{B3}. However, how to do this for $p\in(1,3/2]$ in the discrete
case is not obvious. There are two ingredients, which were employed in
\cite{B3}, that seem to fail in the discrete case. Firstly, it is not
clear if there is a satisfactory counterpart of the theory of
fractional integration in the discrete setup. The idea of fractional
integration was very fruitful and strongly exploited in \cite{Car1},
\cite{Mul1} and \cite{B3}.  Secondly, in \cite{Mul1} and \cite{B3} one
of the key points is based on the dimension-free estimates for the
Riesz transforms \cite{SteinRiesz}. However, Lust--Piquard
\cite{Lu_Piqu1} proved that the discrete Riesz transforms, which
naturally arise in the context of the discrete Laplacian on $\ZZ^d$,
do not have dimensions-free bounds on $\ell^p(\ZZ^d)$ for
$p\in(1, 2)$. In the Appendix we quantify the failure of this dimension-free dependence. 

A similar question concerns the estimates of $r$-variations for the
operators $\mathcal M_t^Q$. We would like to know whether inequality
\eqref{eq:119} can be extended to $p\in(1,3/2]$ or $p\in[4,
\infty)$. Here the situation is even more complicated since we cannot
interpolate with $p=\infty$ as we did in the case of maximal
estimates, so the case for $p\in[4, \infty)$ must be treated
separately.  However, we know \cite{BMSW1} that the operators $M_t^Q$ over the
cubes in $\RR^d$ given by \eqref{eq:112} do have  the dimension-free estimates for $r$-variations on
$L^p(\RR^d)$  for all $p\in(1, \infty)$ and $r\in(2, \infty)$.

The results of \cite{B3} and \cite{BMSW1}, Theorem \ref{thm:9} and the
counterexample of Lust--Piquard \cite{Lu_Piqu1} are certainly
encouraging to understand the situation better and continue further
study of $\mathcal M_t^Q$, which together with $\mathcal M_t^{B^2}$,
is the most natural setting for the discrete Hardy--Littlewood maximal
functions.  We hope to return to these questions in the near future.

\subsection{Notation}
\begin{itemize}
\item Throughout the whole paper $d\in\NN$ will denote the dimension and
$C>0$ will be an absolute constant which does not depend on the
dimension, however it may change  from line  to line.
\item For two real numbers $A, B$ we will write $A \lesssim_{\delta} B$
($A \gtrsim_{\delta} B$) to say that there is an absolute constant
$C_{\delta}>0$ (which possibly depends on $\delta>0$) such that
$A\le C_{\delta}B$ ($A\ge C_{\delta}B$).  We will write
$A \simeq_{\delta} B$ when $A \lesssim_{\delta} B$ and
$A\gtrsim_{\delta} B$ hold simultaneously.
\item 
Let $\NN:=\{1,2,\ldots\}$ be the set of positive integers and $\NN_0:= \NN\cup\{0\}$.

\item 
The Euclidean space $\RR^d$
is endowed with the standard inner product
\[
x\cdot\xi:=\sum_{k=1}^dx_k\xi_k
  \]
for every $x=(x_1,\ldots, x_d)$ and $\xi=(\xi_1, \ldots,
\xi_d)\in\RR^d$. Sometimes we will also write $\langle x, \xi\rangle:=x\cdot\xi$.
\item We will consider two norms on $\RR^d$. For every $x\in\RR^d$
\begin{align*}
  |x|=|x|_2:=\sqrt{x\cdot x} \qquad \text{and} \qquad
  |x|_{\infty}:=\max_{1\le k\le d}|x_k|.
\end{align*}

\item For a countable set $\mathcal Z$ endowed with the counting measure we will
write for any $p\in[1, \infty]$ that
\[
\ell^p(\mathcal Z):=\{f:\mathcal Z\to \CC: \|f\|_{\ell^p(\mathcal Z)}<\infty\},
\]
where for any $p\in[1, \infty)$ we have
\begin{align*}                    
  \|f\|_{\ell^p(\mathcal Z)}:=\Big(\sum_{m\in\mathcal Z}|f(m)|^p\Big)^{1/p} \qquad \text{and} \qquad
  \|f\|_{\ell^{\infty}(\mathcal Z)}:=\sup_{m\in\mathcal Z}|f(m)|.
\end{align*}
In our case usually $\mathcal Z=\ZZ^d$.

\item Let $(X, \mathcal B(X), \mu)$ be a $\sigma$-finite measure space. Let $p\in[1, \infty]$ and suppose that $(T_t: t\in Z)$ is a family of linear operators such that $T_t$ maps $L^p(X)$ to itself for every $t\in Z\subseteq (0, \infty)$. Then the corresponding maximal function will be denoted by
\[
T_{*, Z}f:=\sup_{t\in Z}|T_tf| \quad \text{for every}\quad f\in L^p(X).
\]
We will abbreviate $T_{*, Z}$ to $T_{*}$ if $Z=(0, \infty)$. We use the convention that $T_{*,\emptyset}=0.$
\item Let $(B_1, \|\cdot\|_{B_1})$ and $(B_2, \|\cdot\|_{B_2})$ be Banach spaces. For a linear or sub-linear operator $T: B_1\to B_2$ its norm is defined by
\[
\|T\|_{B_1\to B_2}:=\sup_{\|f\|_{B_1}\le1}\|T(f)\|_{B_2}.
\]

\item Let $\calF$ denote the Fourier transform on $\RR^d$ defined for any function 
$f \in L^1\big(\RR^d\big)$ and for any $\xi\in\RR^d$ as
$$
\calF f(\xi) := \int_{\RR^d} f(x) e^{2\pi i \sprod{\xi}{x}} {\rm d}x.
$$
If $f \in \ell^1\big(\ZZ^d\big)$ we define the discrete Fourier
transform by setting
$$
\hat{f}(\xi) := \sum_{x \in \ZZ^d} f(x) e^{2\pi i \sprod{\xi}{x}}.
$$
for any $\xi\in \TT^d$, where $\TT^d$ denote $d$-dimensional torus which will be  identified with
$[-1/2, 1/2)^d$.

\item To simplify notation we denote by $\mathcal F^{-1}$ the inverse Fourier transform on $\RR^d$
or the inverse Fourier transform (Fourier coefficient) on the torus $\TT^d$.
It will cause no confusions and the meaning will be always clear from the
context. 
\end{itemize}

\section*{Acknowledgements}
The authors are grateful to the referees for careful reading of the manuscript and useful remarks
that led to the improvement of the presentation.

\section{Proofs of Theorem \ref{thm:8} and Theorem \ref{thm:9}}
\label{sec:2}
In this section we prove Theorem \ref{thm:9}, which shows that the phenomenon of
dimension-free estimates in the discrete setting may be completely
different from that what we have seen so far in the continuous setting.
However, we begin with
the observation which shows that the dimension-free estimates for the
discrete Hardy--Littlewood maximal functions are only interesting  if
the supremum is taken over  small scales. The case  when
the supremum is taken over large scales can be easily deduced from
the corresponding continuous estimates by a
comparison principle described in Theorem \ref{thm:8}.

\subsection{Comparison  principle}
For a closed symmetric convex body $G\subset \RR^d$ we defined in \eqref{eq:cg} the constant
\[
c(G):=\inf\{t>0\colon Q_{1/2}\subseteq tG\},
\]
where $Q_{1/2}=[-1/2,1/2]^d$.  We now prove Theorem \ref{thm:8}, which will allow us to transfer  dimension-free estimates
(for large scales described in terms of $c(G)$) between discrete and continuous settings.

\begin{proof}[Proof of Theorem \ref{thm:8}]
For any $f\in\ZZ^d$ we define its extension $F$ on $\RR^d$ by setting
$$F(x):=\sum_{n\in\ZZ^d}f(n)\ind{Q_{1/2}}(x-n)$$
for every $x\in\RR^d$.
Observe that $F(n)=f(n)$ for $n\in\ZZ^d$ and $F\in L^p(\RR^d)$ if
and only if $f\in\ell^p(\ZZ^d)$ with
$\|F\|_{L^p(\RR^d)}=\|f\|_{\ell^p(\ZZ^d)}$ for every $p\ge1$.

Without loss of generality we assume that $f\ge0$, hence $F\ge0$. We
show that for every $t\ge c(G)d$, $d\ge2$ and every $x\in n+Q_{1/2}$ we have
\begin{align}
  \label{eq:89}
  \mathcal M_t^Gf(n)\le \left(1+\frac{6}{d}\right)^d M_{t+2c(G)}^G F(x).
\end{align}
Clearly, this establishes \eqref{eq:trEu1}.
We now focus on \eqref{eq:89}. For $x\in\RR^d$ we denote by $|x|_G$
the Minkowski norm corresponding to $G\subset \RR^d$, i.e.
$$|x|_G:=\inf\{t>0\colon t^{-1}x\in G\}.$$
Then the formula \eqref{eq:cg} may be rephrased as 
\begin{equation*}
c(G)=\sup_{s\in Q_{1/2}}|s|_G.
\end{equation*}
Assume that  $x\in n+ Q_{1/2}$, then we have
\begin{align}
  \label{eq:41}
  \begin{split}
  \mathcal M_t^Gf(n)&
=\frac{1}{|G_t\cap \ZZ^d|}\sum_{m\in\ZZ^d:|n-m|_{G}\le t}f(m)\\
&=\frac{1}{|G_t\cap \ZZ^d|}\sum_{m\in\ZZ^d: |n-m|_{G}\le
	t}\int_{m+Q_{1/2}}F(s)\dif s\\
&\le \frac{1}{|G_t\cap \ZZ^d|}\int_{|x-s|_{G}\le t+2c(G)}F(s)\dif s,
  \end{split}
\end{align}
since, if $|n-m|_G\le t $ and $|s-m|_{\infty}\le1/2, $ then
$$|x-s|_G\le |n-s|_G+|x-n|_G\le |n-m|_G+|s-m|_G+|x-n|_G\le t+2c(G).  $$   

We claim, for $t> c(G)$, that
\begin{equation}
\label{eq:bcomp1}
|G_{t-c(G)}|\le |G_t\cap \ZZ^d|.
\end{equation}
Indeed, if $|s|_G\le t-c(G)$ and $|s-n|_{\infty}\le 1/2,$ then 
$$|n|_G\le |s|_G+|s-n|_G\le t$$ and consequently, we have 
\begin{align*}
|G_{t-c(G)}|=\sum_{n\in \ZZ^d} \int_{n+Q_{1/2}}\ind{G_{t-c(G)}}(s)\dif s\le \sum_{n\in\ZZ^d: |n|_G\le t} \int_{n+Q_{1/2}}1\dif s\le  |G_t\cap \ZZ^d|.
\end{align*}
Hence, using \eqref{eq:41} and \eqref{eq:bcomp1} we obtain for $n\in \ZZ^d$ and  $x\in n+Q_{1/2}$ that 
\begin{equation}
\label{eq:expcomp}
\begin{split}
  \mathcal M_t^Gf(n)&\le \frac{1}{|G_{t-c(G)}|}\int_{|x-s|_{G}\le t+2c(G)}F(s)\dif s\\
  &=\left(\frac{t+2c(G)}{t-c(G)}\right)^dM_{t+2c(G)}^G F(x)\\
                    &=\left(1+\frac{3c(G)}{t-c(G)}\right)^d M_{t+2c(G)}^G F(x)\\
  &\le\left(1+\frac{3}{d-1}\right)^d M_{t+2c(G)}^G F(x).
\end{split}
\end{equation}
The above implies \eqref{eq:89}, hence, \eqref{eq:trEu1}  is proved. We remark that the only place in the proof where the assumption $t\ge c(G)d$ is used is the last inequality in \eqref{eq:expcomp}.
\end{proof}

As a corollary of Theorem \ref{thm:8} we obtain  dimension-free estimates for the maximal functions over  large scales associated with the Hardy--Littlewood averaging operators $\mathcal M_t^{B^q}$ for $q\in[1, \infty]$, where $B^q$ is a ball as in \eqref{eq:114}.
\begin{Corollary}
\label{cor:trEu1}
Given $p\in(1, \infty]$ and $q\in[1, \infty]$, there is a constant $C_{p,q}>0$ such that  for all $f\in\ell^p(\ZZ^d)$ we have
	\begin{equation*}
	\big\|\sup_{t\ge d^{1+1/q}}|\mathcal M_t^{B^q} f|\big\|_{\ell^p}\leq C_{p,q} \|f\|_{\ell^p},
	\end{equation*}
	and the implied constant $C_{p,q}$ is independent of the dimension $d$.
\end{Corollary}
\begin{proof}
By \cite{Mul1} and \cite{B3} we know that for all $p\in(1, \infty]$ and $q\in[1, \infty]$ 
	$$\big\|M_*^{B^q}\big\|_{L^p(\RR^d)\to L^p(\RR^d)}\le C_{p,q}',$$
	for some constant $C_{p,q}'>0$, which is independent of the dimension. Moreover, a simple calculation shows that $2c(B^q)=d^{1/q}.$ Therefore, applying Theorem \ref{thm:8} we obtain
$$	\big\|\sup_{t\ge d^{1/q+1}}|\mathcal M_t^{B^q} f|\big\|_{\ell^p}\le \big\|\sup_{t\ge c(B^q)d}|\mathcal M_t^{B^q} f|\big\|_{\ell^p}\le C_{p,q}'e^6\|f\|_{\ell^p},$$
for all $p\in(1, \infty]$. This completes the proof of the corollary.
\end{proof}

\subsection{Proof of Theorem \ref{thm:9}}
We fix a sequence 
$1\leq \lambda_1<\ldots<\lambda_d<\ldots<\sqrt{2}$  and recall that 
\begin{align*}
  E_d=\Big\{x\in \RR^d\colon \sum_{j=1}^d \lambda_j^2\, x_j^2\,\le 1 \Big\}
\end{align*}
is the ellipsoid as in \eqref{eq:elip}. We note that $\frac{1}{\sqrt{2}}B^2 \subseteq E_d\subseteq  B^2$, hence
\[
\frac{1}{2}d^{1/2}\le c(E_d)\le d^{1/2}.
\]
Therefore, invoking Theorem \ref{thm:8} and arguing in a similar way as in the proof of Corollary \ref{cor:trEu1} we obtain that for every $p\in(3/2, \infty]$   there is a constant $C_{p}>0$ such that the following inequality 
\begin{align*}
  \big\|\sup_{t\ge d^{3/2}}|\mathcal M_t ^{E_d} f|\big\|_{\ell^p}\leq C_{p}\|f\|_{\ell^p},
\end{align*}
holds for all $f\in\ell^p(\ZZ^d)$, since by \cite{B2} and \cite{Car1} we know, for $p\in(3/2, \infty]$, that
	$$\big\|M_*^{E_d}\big\|_{L^p(\RR^d)\to L^p(\RR^d)}\le C_{p}',$$
	for some constant $C_{p}'>0$, which is independent of the dimension.

        On the other hand we shall show that for all $p\in(1, \infty)$ the constant $\mathcal C_p(E_d)$ defined in \eqref{eq:Cepd} grows logarithmically with the dimension, namely 
\[
\mathcal C_p(E_d)\gtrsim(\log d)^{1/p}
\]
 with the implicit constant, which does not depend on $d\in\NN$.

\begin{proof}[Proof of Theorem \ref{thm:9}]
	Let $e_j$ be the $j$-th basis vector from the standard basis
        in $\RR^d.$ We claim that for every $j\in\{1,\ldots,d\}$ we have
	\begin{equation}
	\label{eq:claimOmegaj}
        \Omega_j:=\lambda_j E_d\cap \ZZ^d=\{0,\pm e_1,\ldots,\pm e_j\}.
        \end{equation}
        As the inclusion
        $\{0,\pm e_1,\ldots,\pm e_j\}\subseteq \Omega_j$ is
        straightforward we only focus on proving the reverse
        inclusion. Take $x\in \lambda_j E_d\cap \ZZ^d$ and note that for
        such $x\in\ZZ^d$ we have
        $\sum_{i=1}^d x_i^2 (\lambda_i^2/\lambda_j^{2})\le 1.$ Since we have assumed that $1\leq \lambda_1<\ldots<\lambda_d<\ldots<\sqrt{2}$ we conclude that
        $|x|_{\infty}\le 1$ and that at most one coordinate of $x$,
        say $x_k,$ is non-zero. Moreover, since
        $\sum_{i=1}^d x_i^2 (\lambda_i^2/\lambda_j^{2})\le 1$ we must
        have $k\le j$. Hence, we have shown that
        $\lambda_j E_d\cap \ZZ^d\subseteq \{0,\pm e_1,\ldots,\pm e_j\}$
        thus justifying the claim \eqref{eq:claimOmegaj}.

For every  $j\in\{1,\ldots,d\}$ and  $x\in\ZZ^d$, let 
        \[
\mathcal K_{\Omega_j}(x):=\mathcal K_{\lambda_j}^{E_d}(x)=\frac{1}{|\Omega_j|}\ind{\Omega_j}(x).
        \]
Our aim will be to construct, for every $p\in(1, \infty)$, a non-zero function $f\in\ell^p(\ZZ^d)$ such that
\begin{align*}
  \big\|\sup_{t>0}|\mathcal K_{t}^{E_d}*f|\big\|_{\ell^p}\ge\big\|\max_{1\le j\le d}|\mathcal K_{\Omega_j}*f|\big\|_{\ell^p}\ge C_p(\log d)^{1/p}\|f\|_{\ell^p}
\end{align*}
for some constant $C_p>0$ which depends only on $p$.  For this purpose let $r\in \NN_0$ be such that
$$2^{r+1}-1\le d< 2^{r+2}-1.$$
With this choice of $r$, since $\sum_{s=0}^r 2^s=2^{r+1}-1$, we  decompose $\ZZ^d$ as follows
	$$\ZZ^d=\bigg(\prod_{s=0}^r \ZZ^{I_s}\bigg)\times \ZZ^{a(r,d)},$$
        where $I_s=\{2^{s}-1, 2^s, \ldots,  2^{s+1}-2\}$ and  $a(r,d)=d-2^{r+1}+1$. For $s\in\{0, \ldots, r\}$ we set
	\begin{equation*}
	A_s:=\Big\{y\in \ZZ^{I_s}\colon \forall_{i\in I_s}\ |y_i|\le 2^d\textrm{ and } \sum_{i\in I_s} y_i\textrm{ is odd}\Big\}.
        \end{equation*}
	Note that $|I_s|=2^s$, hence 
        $$(2^{d+1}+1)^{2^s-1}\cdot 2^d\le |A_s\cap \ZZ^{I_s}|\le (2^{d+1}+1)^{2^s}$$
        and thus $$\frac13 (2^{d+1}+1)^{2^s}\le |A_s\cap \ZZ^{I_s}|\le  (2^{d+1}+1)^{2^s}.$$
        Now for each  $x\in\ZZ^d$ we take 
	$$f(x)=\ind{A_0\times\ldots \times A_r}\otimes \delta_{0}(x),$$
	where $\delta_0$ stands for the Dirac delta at zero in
        ${\ZZ^{a(r,d)}}$.

        Therefore, 
        for all $x\in \ZZ^d$  we have
	\begin{equation}
	\label{eq:de4}
	\begin{split}
	&\max_{1\le j\le d}|\mathcal K_{\Omega_j}*f(x)|\ge 	\max_{0\le s< r}|\mathcal K_{\Omega_{2^{s+1}}}*f(x)|\\&
	\ge 	\max_{0\le s< r}\bigg(\ind{A_0\times\ldots\times A_{s-1}}\otimes \bigg(\frac{1}{|\Omega_{2^{s+1}}|}\sum_{j\in I_s}\ind{A_s\pm e_j} \bigg)\otimes\ind{A_{s+1}\times \ldots \times A_r}\otimes \delta_0\bigg)(x).
	\end{split}
		\end{equation}
		For every $s\in\{0, \ldots, r-1\}$ let 
                $$A_s':=\Big\{y\in\ZZ^{I_s}\colon \forall_{i\in I_s}\ |y_i|\le2^d \textrm{ and } \sum_{i\in I_s}
                y_i\textrm{ is even}\Big\}$$ and observe that for all
                $j\in I_s$ we have
                $A_s'\subseteq A_s\pm e_j,$ and
                \begin{align}
                  \label{eq:40}
                  \frac{1}{|\Omega_{2^{s+1}}|}\sum_{j\in I_s}\ind{A_s\pm e_j}(x)
                  =\frac{1}{2^{s+2}+1}\sum_{j\in I_s}\ind{A_s\pm e_j}(x)\ge \frac15 \ind{A_s'}(x);
                \end{align}
		as well as 
                \begin{align*}
                  A_s'\cap A_s=\emptyset\quad\textrm{and}\quad
                  \frac13 (2^{d+1}+1)^{2^s}\le |A_s'\cap \ZZ^{I_s}| \le  (2^{d+1}+1)^{2^s}.  
                \end{align*}
                In particular, for $s\in\{0,\ldots,r-1\}$, the sets
                $$B_s:=A_0\times\ldots \times A_{s-1}\times A_s'\times A_{s+1}\times \ldots\times A_r,$$
		 are pairwise disjoint subsets of $\prod_{s=0}^r \ZZ^{I_s}=\ZZ^{2^{r+1}-1}$ such that
		 \begin{equation}
		 \label{eq:elip_cA}
		 |B_s\cap \ZZ^{2^{r+1}-1}|\ge\frac13 \prod_{s=0}^{r}|A_s\cap \ZZ^{I_s}|
                 = \frac13 \|f\|_{\ell^p}^p.
                 \end{equation}  
		Having defined the sets $B_s$ and using their disjointness and \eqref{eq:40} it follows,  for all $x\in \ZZ^d$, that
		\begin{align*}
                &\max_{0\le s< r}\bigg(\ind{A_0\times\ldots\times A_{s-1}}\otimes
                \bigg(\frac{1}{|\Omega_{2^{s+1}}|}\sum_{j\in I_s}\ind{A_s\pm e_j} \bigg)\otimes\ind{A_{s+1}\times \ldots \times A_r}\otimes \delta_0\bigg)(x)\\
		&\ge \frac{1}{5}\max_{0\le s< r}(\ind{B_s}\otimes \delta_0)(x)
                =\frac{1}{5}\Big(\sum_{s=0}^r (\ind{B_s}\otimes \delta_0)(x)\Big)^{1/p}.
                \end{align*}
                Thus,  by \eqref{eq:de4} and  \eqref{eq:elip_cA} we obtain
                \begin{align*}
                  \big\|\max_{1\le j\le d}|\mathcal K_{\Omega_j}*f|\big\|_{\ell^p}\ge3^{-1/p}5^{-1}r^{1/p}\|f\|_{\ell^p}\ge C_p(\log  d)^{1/p}\|f\|_{\ell^p}
                \end{align*}
                for some constant $C_p>0$ which depends only on $p$.  This completes the  proof of Theorem \ref{thm:9}.
\end{proof}

\section{Fourier transform estimates}
\label{sec:3}
In this section we turn to the main positive results of this paper and only treat the case of cubes. 
We supply estimates independent of the
dimension for the Fourier multipliers $\mathfrak m_t^Q=\hat{\mathcal K}_t^Q$ corresponding
to the operators $\mathcal M_t^Q$ defined in \eqref{eq:0} with $G=Q$,
where $Q=[-1, 1]^d$. In what follows,  the product structure of the cubes $Q_t\cap\ZZ^d$ for $t>0$ will be crucial.
It allows us to prove the key inequalities in Proposition \ref{pro:FourEst}, which are very reminiscent of corresponding inequalities for the continuous case in \cite{B1}, \cite{B2} and \cite{B3}.

From now on, we will only be working with the cubes, so we shall abbreviate
\[
\mathcal M_t=\mathcal M_t^Q, \qquad\mathcal K_t=\mathcal K_t^Q,\qquad \mathfrak m_t^Q=\mathfrak m_t=\hat{\mathcal K}_t.
\]
Note that  $|Q_t\cap\ZZ^d|=|Q_{\lfloor t\rfloor}\cap\ZZ^d|$ and $Q_t\cap\ZZ^d=Q_{\lfloor t\rfloor}\cap\ZZ^d$ for all $t\in(0, \infty)$. Thus
\[
\mathcal K_t(x)=\mathcal K_{\lfloor t\rfloor}(x)=\frac{1}{(2\lfloor t\rfloor+1)^d}\sum_{m\in Q_{\lfloor t\rfloor}}\delta_m(x)\qquad \text{for}\qquad x\in \ZZ^d,
\]
and
\begin{align*}
\mathfrak m_t(\xi)=\mathfrak m_{\lfloor t\rfloor}(\xi)=\frac{1}{(2\lfloor t\rfloor+1)^d}\sum_{m\in Q_{\lfloor t\rfloor}}e^{2\pi i m \cdot\xi
}\qquad \text{for}\qquad \xi\in\TT^d.
\end{align*}
For  $\xi=(\xi_1,\ldots,\xi_d)\in\RR^d$, by a simple calculation, we have
\[
\mathfrak m_t(\xi)=\frac{1}{(2\lfloor t\rfloor+1)^d}\sum_{m\in
  Q_{\lfloor t\rfloor}}e^{2\pi i m \cdot\xi
}=\prod_{k=1}^d\frac{\sin((2\lfloor t\rfloor+1)\pi\xi_k)}{(2\lfloor
  t\rfloor+1)\sin(\pi \xi_k)}.
\]
\begin{remark}
The torus $\TT^d$ is a priori endowed with the   periodic norm
\[
\|\xi\|:=\Big(\sum_{k=1}^d \|\xi_k\|^2\Big)^{1/2}  \qquad \text{for}\qquad \xi=(\xi_1,\ldots,\xi_d)\in\TT^d,
\]
where $\|\xi_k\|:=\dist(\xi_k, \ZZ)$ for all $\xi_k\in\TT$ and  $k\in\{1, \ldots, d\}$.
However, we identify $\TT^d$ with $[-1/2, 1/2)^d$, hence the norm $\|\cdot\|$
coincides with the Euclidean norm $|\cdot|$
 restricted to
$[-1/2, 1/2)^d$. Therefore, throughout this section, unless otherwise
stated, all estimates will be provided in terms of the Euclidean norm
$|\xi|=\big(\sum_{k=1}^d |\xi_k|^2\big)^{1/2}$ for all $\xi\in\TT^d$.
\end{remark}

The main results of this section are gathered in the proposition below.
\begin{proposition}
	\label{pro:FourEst}
	There exists a universal constant $C>0$ such that for every $d\in\NN,$ $t,t_1,t_2\ge 1,$ and
	for every $\xi\in\TT^d$ we have
        \begin{align}
          \label{eq:4}
          \begin{split}
          |\mathfrak m_t(\xi)|&\le\frac{C}{t|\xi|},\\
          |\mathfrak m_t(\xi)-1|&\le Ct|\xi|,\\
          |\mathfrak m_{t_1}(\xi)-\mathfrak m_{t_2}(\xi)|&\le C\big|\lfloor t_1\rfloor-\lfloor t_2\rfloor\big|\max\big\{t_1^{-1},t_2^{-1}\big\}.
          \end{split}
        \end{align}
\end{proposition}

The first estimate in \eqref{eq:4} will follow from Lemma \ref{lem:3} and the remaining two estimates will be a consequence of Lemma \ref{lem:4}. For the proof of Lemma \ref{lem:4} we will need some portion of notations and facts from \cite{B1}.
For every $t\ge0$ and $\xi\in\RR^d$, we introduce
\[
\nu_t(\xi)=\frac{1}{(2t+1)^d}\int_{[-t-1/2, t+1/2]^d}e^{2\pi i
	\xi\cdot x}{\rm d}x.
\]
Changing the variables one obtains
\[
\nu_t(\xi)=\int_{Q_{1/2}}e^{2\pi i
	(2t+1)\xi\cdot x}{\rm d}x=\prod_{k=1}^d\frac{\sin((2t+1)\pi\xi_k)}{(2t+1)\pi \xi_k}.
\]
Observe that $|Q_{1/2}|=1$ and that the cube $Q_{1/2}$ is in the isotropic position, which means that the following equation is satisfied
$$\int_{Q_{1/2}} \langle x, \xi\rangle^2\dif x=L(Q_{1/2})\cdot |\xi|^2\qquad \text{for all}\qquad \xi\in \RR^d,$$
where $L=L(Q_{1/2})$ is the  isotropic constant corresponding to the cube.   

Therefore, it follows from   \cite[eq. (10),(11),(12), p.\ 1473]{B1} (see also \cite[p. 63]{DGM1})  that there is a universal constant
$C>0$ such that for every $d\in\NN$, $t\ge0$ and for every $\xi\in\RR^d$ we have
\begin{align*}
|\nu_t(\xi)|\le C\min\big\{1, (L (2t+1)|\xi|)^{-1}\big\}, \qquad 
|\nu_t(\xi)-1|\le CL (2t+1)|\xi|, \qquad |\langle \xi,\nabla \nu_{0}(\xi)\rangle|\leq C.
\end{align*}
Since it can be easily computed that $L(Q_{1/2})=1/12$, and the above estimates become 
\begin{align}
\label{eq:19}
|\nu_t(\xi)|\le C\min\big\{1, (t|\xi|)^{-1}\big\}, \qquad 
|\nu_t(\xi)-1|\le C(2t+1)|\xi|, \qquad |\langle \xi,\nabla \nu_{0}(\xi)\rangle|\leq C.
\end{align}
Moreover, since $\nu_t(\xi)=\nu_{0}((2t+1)\xi)$ the estimate
$|\langle \xi,\nabla \nu_{0}(\xi)\rangle|\leq C$ and the mean
value theorem give
\begin{align}
\label{eq:91}
|\nu_{N_1}(\xi)-\nu_{N_2}(\xi)|\le C|N_1-N_2|\max\big\{N_1^{-1},N_2^{-1}\big\}
\end{align}
for every $N_1, N_2\in\NN$ and $\xi\in \RR^d$.

The next lemma restates the first bound in \eqref{eq:4}. We remark
that the major difficulty lies in obtaining a dimension-free
estimate. The methods from \cite{B1} do not work in the discrete
setting, but a straightforward argument presented in Lemma \ref{lem:3}
and based on a product structure of $\mathfrak m_N$  leads us to the desired bound.
\begin{lemma}
  \label{lem:3}
  There exists a constant $C>0$ such that for every $d, N\in\NN$ and
  for every $\xi\in\TT^d$ we have
  \begin{align}
    \label{eq:6}
   |\mathfrak m_N(\xi)|= \bigg|\prod_{k=1}^d\frac{\sin((2N+1)\pi\xi_k)}{(2N+1)\sin(\pi \xi_k)}\bigg|\le\frac{C}{N|\xi|}.
  \end{align}
\end{lemma}
\begin{proof}
  The proof will be completed if we show equivalently that there is a
  constant $C>0$  such that for every $d, N\in\NN$ and $0<\xi_k\le
  1/2$ for $k\in\{1,\ldots, d\}$  we have
\begin{align}
    \label{eq:3}
  \sum_{k=1}^d(2N+1)^2\xi_k^2\cdot\bigg(\prod_{j=1}^d\frac{\big(\sin((2N+1)\pi\xi_j)\big)^2}{\big((2N+1)\sin
  (\pi\xi_j)\big)^2}\bigg)\le C.
\end{align}
For $0<|x|\le \pi/2$ we know that
\begin{align}
\label{eq:26}
  \frac{2}{\pi}\le\frac{\sin x}{x}\le 1.
\end{align}
Thus instead of \eqref{eq:3} it suffices to show
\begin{align*}
    \sum_{k=1}^d\big((2N+1)\sin(\pi\xi_k)\big)^2\cdot\bigg(\prod_{j=1}^d\frac{\big(\sin((2N+1)\pi\xi_j)\big)^2}{\big((2N+1)\sin
  (\pi\xi_j)\big)^2}\bigg)\le C.
\end{align*}
For this purpose set
\[
A=\big\{k\in[1, d]\cap\ZZ: \big((2N+1)\sin(\pi\xi_k)\big)^2\ge2\big\}
\]
and note that
\begin{align}
  \label{eq:13}
  \begin{split}
 \sum_{k=1}^d\big((2N+1)\sin(\pi\xi_k)\big)^2&\cdot\bigg(\prod_{j=1}^d\frac{\big(\sin((2N+1)\pi\xi_j)\big)^2}{\big((2N+1)\sin
   (\pi\xi_j)\big)^2}\bigg)\\ &\le\sum_{k\in
   A}\big((2N+1)\sin(\pi\xi_k)\big)^2\cdot\bigg(\prod_{j\in A}\frac{\big(\sin((2N+1)\pi\xi_j)\big)^2}{\big((2N+1)\sin
   (\pi\xi_j)\big)^2}\bigg)\\
 &+
 \sum_{k\in
   A^c}\big((2N+1)\sin(\pi\xi_k)\big)^2\cdot\bigg(\prod_{j\in A^c}\frac{\big(\sin((2N+1)\pi\xi_j)\big)^2}{\big((2N+1)\sin
  (\pi\xi_j)\big)^2}\bigg).
  \end{split}
\end{align}
We shall estimate the sums from \eqref{eq:13} separately. 

For the
first sum let $M=\max_{k\in A}(2N+1)\sin(\pi\xi_k)$. Then
\begin{align*}
  \begin{split}
    \sum_{k\in
   A}\big((2N+1)\sin(\pi\xi_k)\big)^2&\cdot\bigg(\prod_{j\in A}\frac{\big(\sin((2N+1)\pi\xi_j)\big)^2}{\big((2N+1)\sin
   (\pi\xi_j)\big)^2}\bigg)\\
 &\le\sum_{k\in
   A} \big((2N+1)\sin(\pi\xi_k)\big)^2\cdot\frac{1}{M^22^{2|A|-2}}\\
&\le\sum_{k\in
   A} \frac{M^2}{M^22^{2|A|-2}}\le 4|A|\cdot4^{-|A|}\le C.
  \end{split}
\end{align*}

For the second sum  in \eqref{eq:13} we may assume, without  loss of
generality, that $A^c=\{1,\ldots,d\}$. Then it suffices to prove that for any
$0<\xi_k\le \pi/2$ with  $k\in\{1,\ldots,d\}$ we have
\begin{align}
\label{eq:15}  \sum_{k=1}^d\big((2N+1)\sin\xi_k\big)^2\cdot\bigg(\prod_{j=1}^d\frac{\big(\sin((2N+1)\xi_j)\big)^2}{\big((2N+1)\sin
  \xi_j\big)^2}\bigg)\le C,
\end{align}
provided that for all $k\in\{1,\ldots,d\}$ we have $\big((2N+1)\sin\xi_k\big)^2\le2$.
For every $x>0$ we know that
\begin{align*}
  x-\frac{x^3}{3!}<\sin x<x-\frac{x^3}{3!}+\frac{x^5}{5!}.
\end{align*}
This in turn implies that for $0\le x\le 2$ we get
\begin{align}
  \label{eq:25}
  x-\frac{x^3}{6}\le\sin x\le x-\frac{x^3}{8}.
\end{align}
Invoking \eqref{eq:25} twice we obtain
\begin{align}
  \label{eq:11}
\begin{split}
  \sin((2N+1)\xi_k)&\le (2N+1)\xi_k-\frac{\big((2N+1)\xi_k\big)^3}{8}\\
  &\le
  (2N+1)\sin\xi_k+\frac{(2N+1)\xi_k^3}{6}-\frac{\big((2N+1)\xi_k\big)^3}{8}\\
  &=(2N+1)\sin\xi_k-(2N+1)\xi_k^3\bigg(\frac{(2N+1)^2}{8}-\frac{1}{6}\bigg)\\
  &\le (2N+1)\sin\xi_k-\frac{\big((2N+1)\xi_k\big)^3}{10},
\end{split}
\end{align}
since
\[
\frac{(2N+1)^2}{8}-\frac{1}{6}\ge\frac{(2N+1)^2}{10}\Longleftrightarrow(2N+1)^2\ge\frac{20}{3}
\Longleftrightarrow N\ge 1.
\]
Using \eqref{eq:11} and \eqref{eq:26} we see that
\begin{align}
  \label{eq:12}
  \begin{split}
  \frac{\sin((2N+1)\xi_k)}{(2N+1)\sin\xi_k}&\le
                                             1-\frac{\big((2N+1)\xi_k\big)^3}{10(2N+1)\sin\xi_k}\\
                                           &=1-\frac{\big((2N+1)\sin\xi_k\big)^2}{10}\frac{\xi_k^3}{(\sin\xi_k)^3}\\
                                           &\le1-\frac{\big((2N+1)\sin\xi_k\big)^2}{10}.
                                           \end{split}
\end{align}
Using now \eqref{eq:12} we can dominate the left hand side of \eqref{eq:15} and obtain
\begin{align}
  \label{eq:16}
  \begin{split}
  \sum_{k=1}^d\big((2N+1)\sin\xi_k\big)^2&\cdot\bigg(\prod_{j=1}^d\frac{\big(\sin((2N+1)\xi_j)\big)^2}{\big((2N+1)\sin
    \xi_j\big)^2}\bigg)\\
  &\le\sum_{k=1}^d\big((2N+1)\sin\xi_k\big)^2\cdot\prod_{j=1}^d\bigg(1-\frac{\big((2N+1)\sin\xi_j\big)^2}{10}\bigg),
  \end{split}
\end{align}
provided that for all $k\in\{1,\ldots,d\}$ we have $\big((2N+1)\sin\xi_k\big)^2\le2$.
Changing the variables in \eqref{eq:16} by taking $a_k=\big((2N+1)\sin\xi_k\big)^2$ we have to
show that there is a constant $C>0$ such that for any $d\in\NN$ one has 
\begin{align}
  \label{eq:17}
  F(a_1,\ldots, a_d)=(a_1+\ldots+a_d)\prod_{j=1}^d\bigg(1-\frac{a_j}{10}\bigg)\le C
\end{align}
for all $0\le a_k\le 2$ with $k\in\{1,\ldots,d\}$.
To obtain \eqref{eq:17} we note that the estimates
$$1-u\le e^{-u},\qquad ue^{-u}\le 1,$$
which are valid for $u\ge 0,$ lead to
$$ F(a_1,\ldots, a_d)\le (a_1+\ldots+a_d) \exp\left(-\frac{a_1+\ldots+a_d}{10}\right)\le 10.$$
This completes the proof of \eqref{eq:17} and the proof of Lemma \ref{lem:3}. 
\end{proof}

The inequality \eqref{eq:6} immediately implies the first inequality
in \eqref{eq:4}.  We now provide the remaining two inequalities in
\eqref{eq:4}. These are restated as Lemma \ref{lem:4} below. Here we
shall appeal to Lemma \ref{lem:3} and properties
of $\nu_t$ stated in \eqref{eq:19} and \eqref{eq:91}.
  \begin{lemma}
	\label{lem:4}
	There exists a constant $C>0$ such that for every $d, N, N_1, N_2\in\NN$ and
	for every $\xi\in\TT^d$ we have
	\begin{align}
	\label{eq:21}
	|\mathfrak m_N(\xi)-1|\le CN|\xi|,
	\end{align}
	and  
	\begin{align}
	\label{eq:22}
	|\mathfrak m_{N_1}(\xi)-\mathfrak m_{N_2}(\xi)|\le C|N_1-N_2|\max\big\{N_1^{-1},N_2^{-1}\big\}.
	\end{align}
\end{lemma}
\begin{proof}
By \eqref{eq:19} with $t=0$ we have
  \[
\bigg|1-\prod_{k=1}^d\frac{\sin (\pi \xi_k) }{\pi  \xi_k}\bigg|=|1-\nu_0(\xi)|\leq C
  |\xi|.
  \]
Therefore using \eqref{eq:6} we obtain
\begin{align}
  \label{eq:23}
  \begin{split}  
   |\mathfrak m_N(\xi)-\nu_N(\xi)|&= \bigg|\frac{1}{(2N+1)^d}\bigg(\sum_{n\in Q_N}e^{2\pi i
      \xi\cdot n}-\int_{[-N-1/2, N+1/2]^d}e^{2\pi i
      \xi\cdot x}{\rm d}x\bigg)\bigg|\\
    &=\bigg|\frac{1}{(2N+1)^d}\sum_{n\in Q_N}\bigg(e^{2\pi i
    	\xi\cdot n}-\int_{[n-1/2, n+1/2]^d}e^{2\pi i
      \xi\cdot x}{\rm d}x\bigg)\bigg|\\
    &= \bigg|\frac{1}{(2N+1)^d}\sum_{n\in Q_N}e^{2\pi i
      \xi\cdot n}\bigg|\cdot\bigg|1-\prod_{k=1}^d\frac{\sin(\pi \xi_k) }{\pi  \xi_k}\bigg|\\
    &\le C\min\big\{1, \big(N|\xi|\big)^{-1}\big\} |\xi|.
  \end{split}
  \end{align}
  Hence \eqref{eq:21} follows, since by \eqref{eq:23} and \eqref{eq:19} with $t=N$ we get
  \[
|\mathfrak m_N(\xi)-1|\le |\mathfrak m_N(\xi)-\nu_N(\xi)|+|\nu_N(\xi)-1|\le CN|\xi|.
    \]
    
To prove \eqref{eq:22} we will use
  \eqref{eq:91}  and \eqref{eq:23}. We may assume that $N_1\not= N_2$, otherwise there is nothing to do.
     Then 
  \begin{align*}
  |\mathfrak m_{N_1}(\xi)-\mathfrak m_{N_2}(\xi)|&\le|\mathfrak m_{N_1}(\xi)-\nu_{N_1}(\xi)|+
  |\mathfrak m_{N_2}(\xi)-\nu_{N_2}(\xi)|+|\nu_{N_1}(\xi)-\nu_{N_2}(\xi)|\\
  &\le C\sum_{j=1}^2\min\big\{1, \big(N_j|\xi|\big)^{-1}\big\} |\xi|+C|N_1-N_2|\max\big\{N_1^{-1},N_2^{-1}\big\}\\
  &\le 3C|N_1-N_2|\max\big\{N_1^{-1},N_2^{-1}\big\}.
  \end{align*}
  and the proof of Lemma \ref{lem:4} is completed.
  \end{proof}

  \section{Maximal estimates: proofs of Theorem \ref{thm:1} and  Theorem \ref{thm:0}}
\label{sec:4}
In this section we will be concerned with proving Theorem \ref{thm:1}
and Theorem \ref{thm:0}.  Both of the theorems will be a consequence
of a variant of an almost orthogonality principle, which was used by
Carbery \cite[Theorem 2]{Car1} to prove \eqref{eq:113} for
$p\in(3/2, \infty]$.  In Proposition \ref{prop:2} we will adjust the concept from \cite[Theorem 2]{Car1} to the
discrete setup, nevertheless the main idea remains
the same. These ideas will also be employed  in the next section to
estimate $r$-variations.  We begin with the proof of Theorem
\ref{thm:0}, which is simpler. In fact, we prove a stronger result
which will work for any lacunary sequence.\footnote{A sequence
  $(a_n: n\in\NN_0)\subseteq(0, \infty)$ is called lacunary, if
  $a:=\inf_{n\in\NN_0}\frac{a_{n+1}}{a_n}>1$.}

 \begin{theorem}
  \label{thm:0'}
  Let $(a_n: n\in\NN_0)\subseteq(0, \infty)$ be a lacunary sequence. Then for every $p\in(1, \infty]$ there exists a constant $C_p>0$ such that for
  every $d\in\NN$ and every $f\in\ell^p(\ZZ^d)$ we have
  \begin{align}
    \label{eq:87}
\big\|\sup_{n\in\NN_0}|\mathcal M_{a_n}f|\big\|_{\ell^p}=    \big\|\sup_{n\in\NN_0}|\mathcal M_{a_n}^Qf|\big\|_{\ell^p}\le  C_p\|f\|_{\ell^p}.
  \end{align}
  The implied constant $C_p$ may also depend on the quantity $a:=\inf_{n\in\NN_0}\frac{a_{n+1}}{a_n}>1$, which  corresponds to the lacunary sequence $(a_n: n\in\NN_0)$.      
\end{theorem}

We note that by passing to a denser sequence, by a suitable completion of gaps in the underlying sequence, we can assume that the lacunary sequence  $(a_n: n\in\NN_0)$ in Theorem \ref{thm:0'} satisfies additionally an upper bound and is defined on $\ZZ$. Namely, in the rest of the paper, we will assume for all $n\in\ZZ$   that   
\begin{align}
  \label{eq:47}
1<  a\le\frac{a_{n+1}}{a_n}\le a^2.
\end{align}
 
In order to prove our maximal and variational results  we have to construct a suitable
semigroup on $\ZZ^d$, which will be adjusted to our problems. This will be provided in the next paragraph.

\subsection{A diffusion semigroup and corresponding Littlewood--Paley theory}
For every $t\ge0$ let $P_t$ be the Poisson semigroup on
  $\ZZ^d$, which is a convolution operator  defined on the Fourier transform side by the multiplier
  \[
\mathfrak p_t(\xi)=e^{-t|\xi|_{\st}}
    \]
  for every $\xi\in\TT^d$, where 
  \begin{equation*}
  |\xi|_{\st}=\Big(\sum_{k=1}^d (\sin (\pi \xi_k))^2\Big)^{1/2}.
  \end{equation*}
  By  \eqref{eq:26}, for every $\xi\in\TT^d\equiv[-1/2, 1/2)^d$,  we have
  \begin{align}
    \label{eq:92}
    |\xi|\leq |\xi|_{\st}\leq \pi|\xi|.
  \end{align}    
 Let $\{e_1, \ldots, e_d\}$ be the standard basis in $\ZZ^d$. For every $k\in\{1,\ldots, d\}$ and $x\in \ZZ^d$  let
  $$\Delta_k f(x)=f(x)-f(x+e_k)$$
  be the discrete partial  derivative on $\ZZ^d$. The adjoint of $\Delta_k$ is given by $\Delta_k^*f(x)=f(x)-f(x-e_k)$
  and the discrete partial Laplacian is defined as 
  $$\mathcal L_kf(x)=\frac14\Delta_k^* \Delta_kf(x).$$
  
  Then we see that
  \[
\mathcal L_k f(x)=\frac12f(x)-\frac14\big(f(x+e_k)+f(x-e_k)\big),
    \]
   and for every $\xi\in\TT^d$ we obtain
  $$\widehat{(\mathcal L_k f)}(\xi)=\frac{1-\cos (2\pi \xi_k)}{2}\hat{f}(\xi)=(\sin(\pi
  \xi_k))^2\hat{f}(\xi).$$
  
For every $x\in\ZZ^d$ we introduce the maximal function 
 $$P_{*}f(x)=\sup_{t>0}P_t|f|(x)$$
 and the square function
$$g(f)(x)=\left(\int_{0}^{\infty}t\Big|\frac{\rm d}{{\rm d}t} P_t
  f(x)\Big|^2{\rm d}t\right)^{1/2}$$
associated
with the Poisson semigroup $P_t$.

  \begin{lemma}
  	\label{lem:b1}
        For every $p\in(1,  \infty)$ there exists a constant $C_p>0$, which
        does not depend on $d\in\NN$, such that for every $f\in \ell^p(\ZZ^d)$
        we have
\begin{align}
\label{eq:95}
  \|P_{*}f\|_{\ell^p}\leq C_p\|f\|_{\ell^p}  
\end{align}
and
\begin{align}
  \label{eq:96}
  \|g(f)\|_{\ell^p}\leq C_p\|f\|_{\ell^p}.
\end{align}  	
\end{lemma}
For the proof we will have to check that $(P_t: t\ge0)$ is a symmetric diffusion semigroup in the sense of \cite[Chapter III]{Ste1}. For the convenience of the reader we recall the definition of a symmetric diffusion semigroup from \cite[Chapter III, p.65]{Ste1}.  Let $(X, \mathcal B(X), \mu)$ be a $\sigma$-finite measure space. Let $(T_t: t\ge0)$ be a strongly continuous semigroup on $L^2(X)$ which maps $\bigcup_{1\le p\le \infty}L^p(X)$ to $\bigcup_{1\le p\le \infty}L^p(X)$ for every $t\ge0.$ We say that $(T_t: t\ge0)$ is a symmetric diffusion semigroup, if it satisfies for all $t\ge0$ the following conditions: 
\begin{itemize}
\item[(1)] Contraction property: for all $p\in[1, \infty]$ and $f\in L^p(X)$ we have
$\|T_tf\|_{L^p}\le \|f\|_{L^p}$.
\item[(2)] Symmetry property: each $T_t$ is a self-adjoint operator on $L^2(X)$.
\item[(3)] Positivity property: $T_tf\ge0$ if $f\ge0$.
\item[(4)] Conservation property: $T_t1=1$. 
\end{itemize}
  \begin{proof}[Proof of Lemma \ref{lem:b1}] By definition $P_t$ satisfies the semigroup property on $\ell^2.$ Moreover, it is easy to check, working on the Fourier transform side, that for every $f\in\ell^2(\ZZ^d)$ we have
  	\begin{align*}
  	\lim_{t\to0}\|P_tf-f\|_{\ell^2}=0.
  	\end{align*}
 We shall now justify that  $(P_t: t\ge0)$ satisfies conditions (1)-(4) (in particular  the contraction property (1) will ensure that $P_t f \in \ell^p(\ZZ^d)$ if $f\in \ell^p(\ZZ^d)$).  Then, using the  general theory of semigroups from  \cite{Ste1} we obtain \eqref{eq:95} and \eqref{eq:96} with some constant $C_p>0$
 independent of the dimension. 
  	 
    We first note that each $\mathcal L_k$ for $k\in\{1, \ldots, d\}$
    generates a symmetric diffusion semigroup. Indeed,
    $2\mathcal L_k=I-\calP_k,$ where
\[
\calP_kf(x)=\frac12\big(f(x+e_k)+f(x-e_k)\big).
\]
Hence, $\mathcal L_k$ is  self-adjoint on $\ell^2(\ZZ^d)$ and 
\[
e^{-t\mathcal L_k}f=e^{-t/2}\sum_{n=0}^{\infty}\frac{(t/2)^n}{n!}(\calP_k)^n f\quad \text{for all}\quad t\ge0.
\]
This formula obviously yields that $e^{-t\mathcal L_k}$ is
self-adjoint on $\ell^2(\ZZ^d)$, that $e^{-t\mathcal L_k}$ is positive,
and that $e^{-t\mathcal L_k} 1=1$.  Finally, we also deduce the contraction
property for $e^{-t\mathcal L_k}$, since
\[
\|e^{-t\mathcal L_k}f\|_{\ell^p}\leq e^{-t/2}\,\sum_{n=0}^{\infty}\frac{(t/2)^n}{n!}\|(\calP_k)^n f\|_{\ell^p}\leq \|f\|_{\ell^p}.
\]
Summarizing $(e^{-t\mathcal L_k}: t\ge0)$ is a symmetric diffusion
semigroup. Since the operators $\mathcal L_1, \ldots, \mathcal L_d$
are pairwise commuting we have 
\[
\exp(-t(\mathcal L_1+\ldots + \mathcal L_d))=\exp(-t\mathcal L_1)\circ \cdots \circ \exp(-t\mathcal L_d) \quad \text{for all}\quad t\ge0.
\]
Thus $\mathcal L=\mathcal L_1+\ldots +\mathcal L_d$ generates a symmetric diffusion semigroup.
Using the spectral theorem one can easily obtain the subordination formula
\[
e^{-t\mathcal L^{1/2}}=\int_0^\infty e^{-\frac{t^2}{4s}\mathcal
     L} \,(\pi s)^{-1/2}e^{-s}\dif s\quad \text{for all}\quad t\ge0.
\]
From the above we see that
   $(e^{-t\mathcal L^{1/2}}:t\ge0)$ is a symmetric diffusion semigroup as well. In particular we have
   \begin{equation}
   \label{eq:contrLa}
   \|e^{-t\mathcal L^{1/2}}f\|_{\ell^p}\leq \|f\|_{\ell^p}.
   \end{equation}
  	
   Now, in order to complete the proof of the lemma it suffices to note that
\[
P_tf=e^{-t \mathcal L^{1/2}}f,
\]
 since for every $\xi\in\TT^d$ we have
\[
\reallywidehat{\big( e^{-t\mathcal L^{1/2}}f\big)}(\xi)=e^{-t|\xi|_{\st}}\hat{f}(\xi)=\widehat{(P_{t}f)}(\xi).
\]
Applying now the  maximal theorem for semigroups \cite[Chapter
III, Section 3, p.73]{Ste1} together with the Litllewood--Paley theory
for semigroups  \cite[Chapter IV, Section 5, p.111]{Ste1} we
obtain \eqref{eq:95} and \eqref{eq:96} respectively. This completes the proof of Lemma \ref{lem:b1}.
    \end{proof}

Finally, we will need a discrete variant of the  Littlewood--Paley
inequality. For the lacunary sequence $(a_n:n\in \ZZ)$ as in \eqref{eq:47} we define the Poisson projections
$S_n$ by setting
\begin{align*}
S_n=P_{a_n}-P_{a_{n-1}}  
\end{align*}
 for every $n\in\ZZ$.
 Then, clearly for every $f\in \ell^2(\ZZ^d)$, we have
 \begin{align}
   \label{eq:97}
f=\sum_{n\in \ZZ} S_n f.   
 \end{align}
 Using \eqref{eq:96} we show that for every
 $p\in(1, \infty)$ there is $C_p>0$ independent of  $d\in\NN$ such that for every $f\in \ell^p(\ZZ^d)$ we have the following  Littlewood--Paley estimates
 \begin{align}
   \label{eq:98}
   \bigg\|\Big(\sum_{n\in \ZZ}|S_n f|^2\Big)^{1/2}\bigg\|_{\ell^p}\leq C_pa\|f\|_{\ell^p}
 \end{align}
with $a>1$ as in \eqref{eq:47}. In order to establish
 \eqref{eq:98} it suffices to observe that
 $$S_nf(x)=\int_{a_{n-1}}^{a_n}\frac{\rm d}{{\rm d}t} P_t f(x){\rm
   d}t.$$
 Thus by the Cauchy--Schwarz inequality we obtain for every
 $n\in\ZZ$ and $x\in\ZZ^d$ the following bound
\begin{align*}
|S_nf(x)|^2&\leq \bigg(\int_{a_{n-1}}^{a_n}\Big|\frac{\rm d}{{\rm d}t} P_t
  f(x)\Big|{\rm d}t\bigg)^2\le (a_n-a_{n-1})\int_{a_{n-1}}^{a_n}\Big|\frac{\rm d}{{\rm d}t} P_t
  f(x)\Big|^2{\rm d}t \\
  &\le (a^2-1)a_{n-1}\int_{a_{n-1}}^{a_n}\Big|\frac{\rm d}{{\rm d}t} P_t
  f(x)\Big|^2{\rm d}t\le (a^2-1)\int_{a_{n-1}}^{a_n}t\Big|\frac{\rm d}{{\rm d}t} P_t
  f(x)\Big|^2{\rm d}t,
\end{align*}
Now summing over $n\in\ZZ$ and invoking \eqref{eq:96}  we obtain
\eqref{eq:98} and we are done.

\subsection{An almost orthogonality principle}
We now adapt  an almost orthogonality principle from \cite{Car1} for our purposes. The proofs of Theorem \ref{thm:0'} and Theorem \ref{thm:1} will be based on Proposition \ref{prop:2}.
\begin{proposition}
	\label{prop:2}
	Let $(T_{t}: t\in U)$ be a family of linear operators
	defined on $\bigcup_{1\le p\le \infty}\ell^p(\ZZ^d)$ for some index
	set $U\subseteq (0, \infty)$. Suppose that $T_t=M_t-H_t$ for each $t\in U$, where $M_t,H_t$ are positive
	linear operators\footnote{We say that a linear operator $T$ is
		positive if $Tf\ge0$ for every function $f\ge0$.}.   Assume that the following conditions are satisfied.
	\begin{itemize}
		\item  For every $p\in(1, 2]$ we have
		\begin{align}
		\label{eq:57}
		\|H_{*, U}\|_{\ell^p\to\ell^p}<\infty.
		\end{align}
		\item There is $p_0\in(1, 2)$ with the property that for every $p\in (p_0, 2]$ we have
		\begin{align}
		\label{eq:58}
		\sup_{n\in\ZZ}\|T_{*, U_n}\|_{\ell^p\to\ell^p}<\infty,
		\end{align}
		where $U_n=[a_n, a_{n+1})\cap U$ and  $(a_n:{n\in \ZZ})\subseteq(0, \infty)$ is a lacunary sequence obeying \eqref{eq:47}.
		\item There exists a sequence $(b_j: j\in
		\ZZ)$ of positive numbers so that
		$\sum_{j\in\ZZ}b_j^{\rho}=B_{\rho}<\infty$ for every $\rho>0$. Moreover,  for every
		$j\in\ZZ$ we have
		\begin{align}
		\label{eq:60}
		\sup_{\|f\|_{\ell^2}\le1} \bigg\|\Big(\sum_{n\in \ZZ}\sup_{t\in U_n}|T_tS_{j+n}
		f|^2\Big)^{1/2}\bigg\|_{\ell^2}\leq b_j,
		\end{align}
		where $(S_n: n\in\ZZ)$ is the resolution of identity  satisfying \eqref{eq:97} and \eqref{eq:98} for all $p\in(1, \infty)$.
	\end{itemize}
	Then for every $p\in(p_0, 2]$, there exists a constant $\mathbf C_{p}>0$ such
	that
	\begin{align}
	\label{eq:52}
	\sup_{\|f\|_{\ell^p}\le1} \bigg\|\Big(\sum_{n\in \ZZ}\sup_{t\in U_n}|T_t
	f|^2\Big)^{1/2}\bigg\|_{\ell^p}\leq \mathbf C_{p}.
	\end{align}
	The implied constant $\mathbf C_{p}$ depends on the parameters $p, p_0, \rho, B_{\rho}$, the quantities \eqref{eq:57} and  \eqref{eq:58} and the constants $a>1$ and $C_p>0$ as in \eqref{eq:98}. Therefore, $\mathbf C_{p}$ is independent of the dimension as long as the underlying parameters do not depend on $d\in\NN$.  In particular, 
	\begin{align}
	\label{eq:61}
	\|T_{*, U}\|_{\ell^p\to\ell^p}\le \mathbf C_{p}.
	\end{align}
\end{proposition}
\begin{proof}
	Fix $p\in(p_0, 2)$. We note that \eqref{eq:52} immediately implies \eqref{eq:61}. To prove \eqref{eq:52} let $\mathfrak T$ be a family of all possible sequences  $\mathfrak t=(t_n: n\in\ZZ)$  such that
	each component  is a function $t_n:\ZZ^d\to U_n$. For each $N\in\NN_0$ and each $\mathfrak t\in\mathfrak T$ we define a linear operator
	$R_{N}^{\mathfrak t}:\ell^p\to \ell^p(\ell^2)$ by setting
	\begin{align*}
	R_N^{\mathfrak t}f=
	\begin{cases}
	T_{t_n}f, &\text{ if } n\in[-N, N]\cap\ZZ,\\
	\ \ 
	0, &\text{ otherwise}.
	\end{cases}
	\end{align*}
	We observe that $\|R_N^{\mathfrak t}\|_{\ell^p\to \ell^p(\ell^2)} \le(2N+1)\sup_{|n|\le N}\|T_{*, U_n}\|_{\ell^p\to \ell^p}<\infty$ for all $p\in(p_0, 2]$, by \eqref{eq:58}. Our aim will be to show that there is a constant $\mathbf C_{p}>0$ such that
	\begin{align}
	\label{eq:54}
	\sup_{N\in\NN_0}\sup_{\mathfrak t\in \mathfrak T}\|R_N^{\mathfrak t}\|_{\ell^p\to \ell^p(\ell^2)}\le \mathbf C_{p}.
	\end{align}
	Assuming momentarily \eqref{eq:54}  we pick a sequence $\mathfrak t_f=(t_{n, f}: n\in\ZZ)$ where each component $t_{n, f}:X\to U_n$ is a function such that $T_{t_{n, f}}f(x)=\sup_{t\in U_n}|T_tf(x)|$ and obtain
	for every  $N\in\NN_0$ that
	\[
	\bigg\|\Big(\sum_{|n|\le N}\sup_{t\in U_n}|T_tf|^2\Big)^{1/2}\bigg\|_{\ell^p}=\bigg\|\Big(\sum_{|n|\le N}|T_{t_{n, f}}f|^2\Big)^{1/2}\bigg\|_{\ell^p}\le \sup_{\mathfrak t\in\mathfrak T}\|R_{N}^{\mathfrak t}\|_{\ell^p\to \ell^p(\ell^2)}\|f\|_{\ell^p}\le \mathbf C_p\|f\|_{\ell^p},
	\]
	where the last inequality follows from \eqref{eq:54}. Then invoking the monotone convergence theorem we obtain the claim in \eqref{eq:52}. To prove \eqref{eq:54} we fix $\mathfrak t=(t_{n}: n\in\ZZ)\in\mathfrak T$ and
	for $s\in(p_0, 2]$ and $r\in [1, \infty]$  let  $A_N(s, r)$ be the best constant in the following inequality
	\begin{align}
	\label{eq:55}
	\bigg\|\Big(\sum_{|n|\le N}|T_{t_n}g_n
	|^r\Big)^{1/r}\bigg\|_{\ell^s}\le A_N(s, r)\bigg\|\Big(\sum_{|n|\le N}|g_n
	|^r\Big)^{1/r}\bigg\|_{\ell^s}.
	\end{align}
	Using \eqref{eq:58} it is easy to see that $A_N(s, r)<\infty$.
	We pick a real number $q$ such that  $p_0<q<p<2$ and define $\theta\in(0, 1)$ by setting
	\[
	\frac{1}{2}=\frac{1-\theta}{q}+\frac{\theta}{\infty}.
	\]
	This in turn implies that $\theta=1-q/2$ and consequently
	determines $u\in(q, p)$ such that
	\[
	\frac{1}{u}=\frac{1-\theta}{q}+\frac{\theta}{p}.
	\]
	Using the complex method of interpolation for the $\ell^s(\ell^r)$ spaces, see e.g.\ \cite[Theorem 5.1.2]{BeLo}, we obtain
	$$A_N(u, 2)\le A_N(q, q)^{1-\theta}A_N(p, \infty)^{\theta}.$$
	Invoking \eqref{eq:58} we have $A_N(q, q)\le\sup_{n\in\ZZ}\|T_{*, U_n}\|_{\ell^q\to\ell^q}$, since
	\begin{align*}
	\bigg\|\Big(\sum_{|n|\le N}|T_{t_n}g_n
	|^q\Big)^{1/q}\bigg\|_{\ell^q}\le\bigg(\sum_{|n|\le N}\big\|\sup_{t\in U_n}|T_tg_n
	|\big\|_{\ell^q}^q\bigg)^{1/q}\le \sup_{n\in\ZZ}\|T_{*, U_n}\|_{\ell^q\to\ell^q}\bigg\|\Big(\sum_{|n|\le N}|g_n
	|^q\Big)^{1/q}\bigg\|_{\ell^q}.  
	\end{align*}
	Invoking \eqref{eq:57} we obtain $A_N(p, \infty)\le \|R_N^{\mathfrak t}\|_{\ell^p\to \ell^p(\ell^2)}+2\|H_{*, U}\|_{\ell^p\to \ell^p}$. Indeed, 
	let $g=\sup_{|n|\le N}|g_n|$ and recall that $M_t=T_t+H_t$, then
	\begin{align*}
	\big\|\sup_{|n|\le N}|T_{t_n}g_n
	|\big\|_{\ell^p}&\le\big\|\sup_{|n|\le N}M_{t_n}g
	\big\|_{\ell^p}+\big\|\sup_{t\in U}H_tg
	\big\|_{\ell^p}\\
	&\le \bigg\|\Big(\sum_{|n|\le N}|T_{t_n}g|^2\Big)^{1/2}\bigg\|_{\ell^p}
	+2\big\|\sup_{t\in U}H_tg
	\big\|_{\ell^p}\\ 
	&\le\big(\|R_N^{\mathfrak t}\|_{\ell^p\to \ell^p(\ell^2)}+2\|H_{*, U}\|_{\ell^p\to \ell^p}\big)\|g\|_{\ell^p}.
	\end{align*}
	Moreover, \eqref{eq:60} implies that
	\begin{align}
	\label{eq:56}
	\bigg\|\Big(\sum_{|n|\le N}|T_{t_n}S_{j+n}f
	|^2\Big)^{1/2}\bigg\|_{\ell^2}\le b_j\|f\|_{\ell^2}.
	\end{align}
	By \eqref{eq:55} and \eqref{eq:98} we get
	\begin{equation}
	\label{eq:65}
	\begin{split}
	\bigg\|\Big(\sum_{|n|\le N}|T_{t_n}S_{j+n}f
	|^2\Big)^{1/2}\bigg\|_{\ell^u}&\le A_N(u, 2)\bigg\|\Big(\sum_{|n|\le N}|S_{j+n}f
	|^2\Big)^{1/2}\bigg\|_{\ell^u}\\
	&\le C_{u}aA_N(u, 2)\|f\|_{\ell^u}.
	\end{split}
	\end{equation}
	We now take $\rho\in(0, 1)$ satisfying
	\[
	\frac{1}{p}=\frac{1-\rho}{u}+\frac{\rho}{2}
	\]
	and interpolate \eqref{eq:56} with \eqref{eq:65}, then
	\begin{align}
	\label{eq:66}
	\bigg\|\Big(\sum_{|n|\le N}|T_{t_n}S_{j+n}f
	|^2\Big)^{1/2}\bigg\|_{\ell^p}\le \big(C_{u}aA_N(u, 2)\big)^{1-\rho}b_j^{\rho}\|f\|_{\ell^p}.
	\end{align}
	Summing \eqref{eq:66} over $j\in\ZZ$  it is easy to see that
	\[
	\|R_N^{\mathfrak t}\|_{\ell^p\to \ell^p(\ell^2)}\le\Big(C_{u}a\sup_{n\in\ZZ}\|T_{*, U_n}\|_{\ell^q\to\ell^q}^{1-\theta}\big(\|R_N^{\mathfrak t}\|_{\ell^p\to \ell^p(\ell^2)}+2\|H_{*, U}\|_{\ell^p\to \ell^p}\big)^{\theta}\Big)^{1-\rho}B_
	{\rho}.
	\]
	Thus there exists $0<\mathbf C_p<\infty$,  such that \eqref{eq:54} holds.
	This completes the proof of Proposition \ref{prop:2}.
\end{proof}

We now  prove Theorem \ref{thm:0'}, which immediately implies Theorem \ref{thm:0} by taking $a_n=2^n$ for all $n\in\NN_0$. 

\subsection{Proof of Theorem \ref{thm:0'}}
To prove Theorem \ref{thm:0'} we shall exploit  Proposition \ref{prop:2} and Proposition \ref{pro:FourEst} and  properties of the Poisson semigroup $P_t$ from Lemma \ref{lem:b1} and the Littlewood--Paley inequality \eqref{eq:98}. 
\begin{proof}[Proof of Theorem \ref{thm:0'}]
Observe first that since $M_{a_n}f=f$ if $a_n<1$ we can assume without loss of generality that our lacunary sequence is such that $ a_{0}\ge 1.$ 	
	
  Inequality \eqref{eq:95} ensures that
  \[
 \big\|\sup_{n\in\NN_0}|\mathcal M_{a_n}f|\big\|_{\ell^p}\le\|P_{*}f\|_{\ell^p}+\big\|\sup_{n\in\NN_0}|(\mathcal M_{a_n}-P_{a_n})f|\big\|_{\ell^p}\le C\|f\|_{\ell^p}+\big\|\sup_{n\in\NN_0}|(\mathcal M_{a_n}-P_{a_n})f|\big\|_{\ell^p}.
    \]
    We only have to handle the second maximal function.  For this
    purpose we will appeal to Proposition \ref{prop:2} with the parameter
    $p_0=1$, the set $U=\{a_n: n\in \NN_0\}$ (so that $U\subseteq [1,\infty)$), the operators $M_t=\mathcal M_t$, and $H_t=P_t$,
     where $P_t$ is the Poisson semigroup and a sequence $b_j\simeq a^{-|j|/2}$, where $a>1$ is the
     parameter from \eqref{eq:47}.

     In this case \eqref{eq:57} and
    \eqref{eq:58} are obvious, since $U_n=[a_n, a_{n+1})\cap U=\{a_n\}$ if $n\ge 0$ and $U_n=\emptyset$ if $n<0$ (recall our convention that $T_{*,\emptyset}=0$). Thus it only remains to verify condition
    \eqref{eq:60}. For every $t>0$ and $\xi\in\TT^d$ let
    \begin{equation}
    \label{eq:kndef}
    \mathfrak n_t(\xi)=\mathfrak m_t(\xi)-\mathfrak p_t(\xi)=\mathfrak m_t(\xi)-e^{-t|\xi|_{\st}},\qquad 
    \end{equation}
    be the multiplier associated with the operator
    $T_t=M_t-H_t=\mathcal M_t-P_t$.  Observe that by Proposition
    \ref{pro:FourEst}, the inequality \eqref{eq:92} and the properties
    of $\mathfrak p_t(\xi)$ there exists a constant $C>0$ independent of the
    dimension such that for $t\ge 1$ we have
  \begin{align}
    \label{eq:99}
    |\mathfrak n_t(\xi)|\le |\mathfrak m_t(\xi)-1|+|\mathfrak p_t(\xi)-1|\le Ct|\xi|, \qquad\text{and} \qquad |\mathfrak n_t(\xi)|\le
    C(t|\xi|)^{-1}. 
  \end{align}
  Therefore, by \eqref{eq:99}, the Plancherel theorem and \eqref{eq:92} we
obtain
\begin{align}
\label{eq:99+}
\begin{split}
    \bigg\|\Big(\sum_{n\in \ZZ}\sup_{t\in U_n}&|T_tS_{j+n}
  f|^2\Big)^{1/2}\bigg\|_{\ell^2}=\bigg\|\Big(\sum_{n\in \NN_0}|T_{a_n}S_{j+n}
  f|^2\Big)^{1/2}\bigg\|_{\ell^2}\\
  &=\bigg(\int_{\TT^d}\sum_{n\in \NN_0}\big|\mathfrak n_{a_n}(\xi)\big(e^{-a_{n+j} |\xi|_{\st}}
  -e^{-a_{n+j-1}|\xi|_{\st}}\big)\big|^2
  |\hat{f}(\xi)|^2{\rm d}\xi\bigg)^{1/2}\\
 & \lesssim\bigg(\int_{\TT^d}\sum_{n\in \NN_0}E_{n,j}(\xi)^2
  |\hat{f}(\xi)|^2{\rm d}\xi\bigg)^{1/2},
\end{split}
\end{align}
where
\[
E_{n,j}(\xi):=\min\big\{a_n|\xi|, (a_n|\xi|)^{-1}\big\}\big|\big(e^{-a_{n+j} |\xi|_{\st}}
-e^{-a_{n+j-1}|\xi|_{\st}}\big)\big|.
\]

We claim that
\begin{equation}
\label{eq:Enjxi}
E_{n,j}(\xi)\le a^{-|j|/2}\min\big\{(a_n|\xi|)^{1/2}, (a_n|\xi|)^{-1/2}\big\}.
\end{equation}
Indeed, if $j\ge0$, then
\begin{align*}
  E_{n,j}(\xi)&\lesssim\min\big\{a_n|\xi|, (a_n|\xi|)^{-1}\big\}
               \cdot  e^{-a_{n+j-1}|\xi|_{\st}}\\
              &\lesssim \min\big\{(a_n|\xi|)^{1/2}, (a_n|\xi|)^{-1/2}\big\}(a_n|\xi|)^{1/2}e^{-{a^{j-1}a_{n}|\xi|_{\st}}}\\
  &\lesssim_a a^{-j/2}\min\big\{(a_n|\xi|)^{1/2}, (a_n|\xi|)^{-1/2}\big\}.
\end{align*}
If $j<0$, then
\begin{align*}
  E_{n,j}(\xi)&\lesssim\min\big\{a_n|\xi|, (a_n|\xi|)^{-1}\big\}\min\big\{a_{n+j}|\xi|,e^{-a_{n+j-1}|\xi|_{\st}}\big\}\\
              &\lesssim \min\big\{(a_n|\xi|)^{1/2}, (a_n|\xi|)^{-1/2}\big\}(a_n|\xi|)^{-1/2}(a_{n+j}|\xi|)^{1/2}\\
  &\lesssim_a a^{j/2}\min\big\{(a_n|\xi|)^{1/2}, (a_n|\xi|)^{-1/2}\big\}.
\end{align*}

We use \eqref{eq:Enjxi} to  estimate \eqref{eq:99+} and obtain
\begin{equation}
\label{eq:99++}
\begin{split}
\bigg\|\Big(\sum_{n\in \ZZ}\sup_{t\in U_n}|T_tS_{j+n}
f|^2\Big)^{1/2}\bigg\|_{\ell^2}
&\lesssim_aa^{-|j|/2}\bigg(\int_{\TT^d}\sum_{n\in
	\ZZ}\min\big\{a_n |\xi|, (a_n|\xi|)^{-1}\big\}
|\hat{f}(\xi)|^2{\rm d}\xi\bigg)^{1/2}\\
&
\lesssim_a a^{-|j|/2}\|f\|_{\ell^2}.
\end{split}
\end{equation}
Note that in \eqref{eq:Enjxi} and \eqref{eq:99++} only the lower bound from \eqref{eq:47} is required. The inequality \eqref{eq:99++} implies condition \eqref{eq:60} in Proposition \ref{prop:2}. Hence, the proof of Theorem \ref{thm:0'} is  completed.
\end{proof}

\subsection{Proof of Theorem \ref{thm:1}}
We shall demonstrate how to use Proposition \ref{prop:2} and Proposition \ref{pro:FourEst} to
deduce Theorem \ref{thm:1}. The new ingredient will be inequality
\eqref{eq:73}, which is invaluable here. The proof of Lemma \ref{lem:6} immediately  follows from  \cite[Lemma 2.1]{BMSW1} for $r=\infty$, hence we omit it here.
\begin{lemma}
  \label{lem:6}
For every  $n\in\NN_{0}$ and   every function $\mathfrak a:[2^n,
        2^{n+1}]\to\CC$ satisfying $\mathfrak a(t)=\mathfrak a(\lfloor t\rfloor)$  we have
        \begin{align}
          \label{eq:73}
            	\sup_{2^n\le t< 2^{n+1}}|\mathfrak a(t)-\mathfrak a(2^n)|
		\le
		2^{1-1/r}
		\sum_{0\le l\le n}
		\Big(
		\sum_{k = 0}^{2^{l}-1}
		\big|\mathfrak a(2^n+{2^{n-l}(k+1)}) -\mathfrak a(2^n+{2^{n-l}k})\big|^r
		\Big)^{1/r}.
        \end{align}
\end{lemma}

\begin{proof}[Proof of Theorem \ref{thm:1}]
Observe that $\mathcal M_t=I$ for $0<t<1.$ Therefore it suffices to
bound $\sup_{t\ge 1}|\mathcal M_t f|$. In a similar way as in Theorem
\ref{thm:0'} we will use Proposition \ref{prop:2} with the parameter
$p_0=3/2$, the sequence $a_n=2^n$, the set $U=[1,\infty)$, the
operators $M_t=\mathcal M_t$ and $H_t=\mathcal M_{2^n}$ for every
$t\in U_n=[2^n, 2^{n+1})\cap[1,\infty)$, and a sequence
$b_j\simeq2^{-\varepsilon|j|/4}$ for some $\varepsilon\in(0, 1)$.

   Theorem \ref{thm:0} ensures that condition
  \eqref{eq:57} holds for the operators $H_t$. It remains to prove \eqref{eq:58} for all
  $p\in(3/2, 2]$ and verify condition \eqref{eq:60} for the operators $T_t=M_t-H_t$.  We first prove
  \eqref{eq:58}, for this purpose
  we use Lemma \ref{lem:6} and  obtain, for $n\ge 0,$
    \begin{align}
      \label{eq:94}
      \begin{split}
        \big\|\sup_{t\in U_n}|T_tf|\big\|_{\ell^p}&=\big\|\sup_{2^n\le
          t<2^{n+1}}|(\mathcal M_t-\mathcal M_{2^n})f|\big\|_{\ell^p}\\
        &\lesssim \sum_{0\le l\le n}\bigg\|\Big( \sum_{k = 0}^{2^{l}-1}
        \big|(\mathcal M_{2^n+{2^{n-l}(k+1)}} - \mathcal M_{2^n+{2^{n-l}k}})f\big|^2
        \Big)^{1/2}\bigg\|_{\ell^p}\lesssim \|f\|_{\ell^p}.
      \end{split}
    \end{align}
    The last inequality in \eqref{eq:94} will follow if we show that
    for every $p\in(3/2, 2]$, there is $\delta_p>0$ such that
    \begin{align*}
      \bigg\|\Big( \sum_{k = 0}^{2^{l}-1}
        \big|(\mathcal M_{2^n+{2^{n-l}(k+1)}} - \mathcal M_{2^n+{2^{n-l}k}})f\big|^2
        \Big)^{1/2}\bigg\|_{\ell^p}\lesssim 2^{-\delta_pl}\|f\|_{\ell^p}.
    \end{align*}
    This in turn will follow by interpolation between \eqref{eq:101}
    and \eqref{eq:102}, since for $p=1$ it is easy to see that
\begin{align}
  \label{eq:101}
      \bigg\|\Big( \sum_{k = 0}^{2^{l}-1}
        \big|(\mathcal M_{2^n+{2^{n-l}(k+1)}} - \mathcal M_{2^n+{2^{n-l}k}})f\big|^2
        \Big)^{1/2}\bigg\|_{\ell^1}\lesssim 2^{l}\|f\|_{\ell^1},
\end{align}
and for $p=2$ we are going to  show that 
\begin{align}
 \label{eq:102}
      \bigg\|\Big( \sum_{k = 0}^{2^{l}-1}
        \big|(\mathcal M_{2^n+{2^{n-l}(k+1)}} - \mathcal M_{2^n+{2^{n-l}k}})f\big|^2
        \Big)^{1/2}\bigg\|_{\ell^2}\lesssim 2^{-l/2}\|f\|_{\ell^2}.
\end{align}
 We have reduced the matter to estimate \eqref{eq:102}, which will be
 based on inequality \eqref{eq:103}. For every $\varepsilon\in[0, 1)$
 we have by \eqref{eq:22} that
 \begin{align}
   \label{eq:103}
   \begin{split}     
   \sum_{k = 0}^{2^{l}-1}
   \big|\mathfrak m_{2^n+2^{n-l}(k+1)}(\xi)-\mathfrak m_{2^n+2^{n-l}k}(\xi)\big|^{2-\varepsilon}
   \lesssim \sum_{k = 0}^{2^{l}-1}\bigg(\frac{2^{n-l}}{2^n+2^{n-l}k}\bigg)^{2-\varepsilon}
   \lesssim \frac{1}{2^{(1-\varepsilon)l}}.
   \end{split}
 \end{align}
Plancherel's theorem and inequality \eqref{eq:103} with
$\varepsilon=0$ yield \eqref{eq:102}, which completes the proof of 
\eqref{eq:94}.

We now verify condition \eqref{eq:60}.  As
before we apply Lemma \ref{lem:6} and for every $\varepsilon\in(0, 1)$
we get
\begin{align}
  \label{eq:104}
  \begin{split}
  \bigg\|\Big(\sum_{n\in \ZZ}\sup_{t\in U_n}&|T_tS_{j+n}
  f|^2\Big)^{1/2}\bigg\|_{\ell^2}=\bigg\|\Big(\sum_{n\in \NN_0}\sup_{t\in U_n}|T_tS_{j+n}
  f|^2\Big)^{1/2}\bigg\|_{\ell^2}\\
  &\lesssim \sum_{l\ge0}
  \bigg\|\Big(\sum_{n\ge l} \sum_{k = 0}^{2^{l}-1} \big|(\mathcal M_{2^n+{2^{n-l}(k+1)}}
  - \mathcal M_{2^n+{2^{n-l}k}})S_{j+n}f\big|^2
  \Big)^{1/2}\bigg\|_{\ell^2}\\
  &\lesssim 2^{-\varepsilon|j|/4}\|f\|_{\ell^2}.
  \end{split}
\end{align}
The last inequality follows from the following inequality 
\begin{align}
  \label{eq:105}
\bigg\|\Big(\sum_{n\ge l} \sum_{k = 0}^{2^{l}-1} \big|(\mathcal M_{2^n+{2^{n-l}(k+1)}}
  - \mathcal M_{2^n+{2^{n-l}k}})S_{j+n}f\big|^2
  \Big)^{1/2}\bigg\|_{\ell^2}
  \lesssim 2^{-(1-\varepsilon)l/2}2^{-\varepsilon|j|/4}\|f\|_{\ell^2}.  
\end{align}
The proof of \eqref{eq:105} will be very much in spirit of \eqref{eq:102}.
Indeed, $(2^n+{2^{n-l}k})\simeq 2^n$ for every $0\le
k\le 2^{l}$, hence due to \eqref{eq:92},  \eqref{eq:6}, and \eqref{eq:21} we obtain
\begin{align}
\label{eq:106}
  \begin{split}
  \big|\mathfrak m_{2^n+2^{n-l}(k+1)}(\xi)&-\mathfrak m_{2^n+2^{n-l}k}(\xi)\big|^{\varepsilon}
\big|\mathfrak p_{2^{j+n}}(\xi)-\mathfrak p_{2^{j+n-1}}(\xi)\big|^2\\
&\lesssim
\min\big\{|2^n\xi|,|2^n\xi|^{-1}\big\}^{\varepsilon}\big|e^{-2^{j+n}|\xi|_{\st}}-e^{-2^{j+n-1}|\xi|_{\st}}\big|^2\\
&\lesssim 2^{-\varepsilon|j|/2}\min\big\{|2^n\xi|,|2^n\xi|^{-1}\big\}^{\varepsilon/2},
  \end{split}
\end{align}
where the last  inequality  follows from \eqref{eq:Enjxi}.
Finally,  \eqref{eq:106} combined with \eqref{eq:103} yields
\begin{align}
  \label{eq:107}
  \begin{split}
  \sum_{k = 0}^{2^{l}-1}
\big|\mathfrak m_{2^n+2^{n-l}(k+1)}(\xi)&-\mathfrak m_{2^n+2^{n-l}k}(\xi)\big|^2
\big|\mathfrak p_{2^{j+n}}(\xi)-\mathfrak p_{2^{j+n-1}}(\xi)\big|^2
  \\
  &\lesssim
2^{-\varepsilon|j|/2}2^{-(1-\varepsilon)l}\min\big\{|2^n\xi|,|2^n\xi|^{-1}\big\}^{\varepsilon/2}.
  \end{split}
\end{align}
Therefore, \eqref{eq:107} with the Plancherel theorem establish
\eqref{eq:105}, since
\begin{align*}
\bigg\|\Big(\sum_{n\ge l} \sum_{k = 0}^{2^{l}-1} &\big|(\mathcal M_{2^n+{2^{n-l}(k+1)}}
- \mathcal M_{2^n+{2^{n-l}k}})S_{j+n}f\big|^2
  \Big)^{1/2}\bigg\|_{\ell^2}^2\\
  &\lesssim2^{-\varepsilon|j|/2}2^{-(1-\varepsilon)l}
  \int_{\TT^d}\sum_{n\ge l}\min\big\{|2^n\xi|,|2^n\xi|^{-1}\big\}^{\varepsilon/2}|\hat{f}(\xi)|^2{\rm d}\xi\\
  &\lesssim 2^{-\varepsilon|j|/2}2^{-(1-\varepsilon)l}\|f\|_{\ell^2}^2.
\end{align*}
This justifies \eqref{eq:104} and completes the proof of Theorem \ref{thm:1}.
   \end{proof}

\section{$r$-variational estimates: proofs of Theorem \ref{thm:11} and Theorem \ref{thm:10}}
We begin with some remarks on $r$-variation seminorms.  For $r\in[1, \infty)$ the
$r$-variation seminorm $V_r$ of a complex-valued function
$(0, \infty)\ni t\mapsto\mathfrak a_t$ is defined by
\[
V_r(\mathfrak a_t: t\in Z)=\sup_{\atop{t_0<\ldots <t_J}{t_j\in Z}}
\bigg(\sum_{j=0}^J|\mathfrak a_{t_{j+1}}-\mathfrak a_{t_j}|^r\bigg)^{1/r},
\]
where the supremum is taken over all finite increasing sequences in $Z\subseteq(0, \infty)$.
\begin{itemize}
\item If $1\le r_1\le r_2<\infty$ then
\[
V_{r_2}(\mathfrak a_t: t\in Z)\le V_{r_1}(\mathfrak a_t: t\in Z).
\]
\item If  $Z_1\subseteq Z_2$ then
\[
V_r(\mathfrak a_t: t\in Z_1)\le V_r(\mathfrak a_t: t\in Z_2).
\]
\item If  $Z$ is a disjoint sum of $Z_1$ and $Z_2$ then
\begin{align}
\label{eq:18-}
V_r(\mathfrak a_t: t\in Z)\le V_r(\mathfrak a_t: t\in Z_1)+ V_r(\mathfrak a_t: t\in Z_2)+2\sup_{t\in Z} |\mathfrak a_t|.
\end{align}
\item For every $t_0\in Z$ we have
\begin{align*}
\sup_{t\in Z}|\mathfrak a_t|\le|\mathfrak a_{t_0}|+ 2V_r(\mathfrak a_t: t\in Z).
\end{align*}
\item If $Z$ is a countable subset of $(0, \infty)$ then
\begin{align*}
V_r(\mathfrak a_t: t\in Z)\le 2\Big(\sum_{t \in Z} |\mathfrak a_t|^r\Big)^{1/r}.
\end{align*}
\end{itemize}
Finally, for every $r\in[1, \infty)$ there exists $C_r>0$ such that	
\begin{equation}
\label{eq:29}
V_r(\mathfrak a_t: t\in Z)\leq C_rV_r(\mathfrak a_{t}: t\in Z\cap\mathbb D)+C_r\Big(\sum_{n\in\ZZ}
V_r\big(( \mathfrak a_t-\mathfrak a_{2^n}): t\in[2^n, 2^{n+1})\cap Z\big)^r\Big)^{1/r},
\end{equation}
where $\mathbb D=\{2^n: n\in\ZZ\}$.
The first  quantity on the right side in \eqref{eq:29} is called the
long variation seminorm, whereas the second is called  the short
variation seminorm. 
This is a very useful inequality which, in view of Theorem
\ref{thm:11}, will allow  us to reduce the proof
of Theorem \ref{thm:10}   to the estimates of short
variations associated with $\mathcal M_t$. 
\subsection{Long variations: proof of Theorem \ref{thm:11}}
We have shown in Section \ref{sec:4} that the semigroup $P_t$ is a symmetric diffusion semigroup. Therefore from \cite[Theorem 3.3]{JR1} (see also \cite[inequality (2.30)]{BMSW1})
we conclude that for every $p\in(1, \infty)$ and for
 every $r\in(2, \infty)$ there is a constant $C_{p, r}>0$ independent
 of the dimension such that for every $f\in \ell^p(\ZZ^d)$ we have
\begin{equation}
\label{eq:2}
\big\|V_r\big( P_{t} f: t>0\big)\big\|_{\ell^p}\leq C_{p, r}\|f\|_{\ell^p}.
\end{equation}

For every $f\in \ell^p(\ZZ^d)$  we obtain 
\begin{align}
\label{eq:30}
  \big\lVert
	V_r\big(  \mathcal M_{ 2^n} f: n\in\NN_0\big)
	\big\rVert_{\ell^p}\le
  \big\lVert
	V_r\big( P_{ 2^n} f: n\in\NN_0\big)
	\big\rVert_{\ell^p}
  +\bigg\|\Big(\sum_{n\in\NN_0}\big| (\mathcal M_{2^n}- 
 P_{2^n})f\big|^2\Big)^{1/2}\bigg\|_{\ell^p}.
\end{align}
The first term in \eqref{eq:30} is bounded on $\ell^p(\ZZ^d)$ by
\eqref{eq:2}. Therefore, it remains to obtain $\ell^p(\ZZ^d)$ bounds for the
square function in \eqref{eq:30}. For this purpose we will use
\eqref{eq:97} with $a_n=2^n$ (so that $a=2$). Indeed, observe that
\begin{align}
\label{eq:31}
  \begin{split}
\bigg\|\Big(\sum_{n\in \NN_0}\big| (\mathcal M_{2^n}&- 
  P_{2^n})f\big|^2\Big)^{1/2}\bigg\|_{\ell^p}
 =\bigg\|\Big(\sum_{n\in \NN_0}\big|\sum_{j\in \ZZ} (\mathcal M_{2^n} - 
 P_{2^n})S_{j+n}f\big|^2\Big)^{1/2}\bigg\|_{\ell^p}\\
 &\le \sum_{j\in \ZZ}\bigg\|\Big(\sum_{n\in \NN_0}\big| (\mathcal M_{2^n} - 
 P_{2^n})S_{j+n}f\big|^2\Big)^{1/2}\bigg\|_{\ell^p}\lesssim
 \sum_{j\in\ZZ}2^{-\delta_p|j|}\|f\|_{\ell^p}\lesssim \|f\|_{\ell^p}.    
  \end{split}
\end{align}
In order to justify the last but one inequality in \eqref{eq:31} it suffices to show, for each $j\in\ZZ$, that
\begin{equation}
\label{eq:32} 
\bigg\|\Big(\sum_{n\in \NN_0}\big|  \mathcal M_{2^n} S_{j+n}f \big|^2\Big)^{1/2}\bigg\|_{\ell^p}+\bigg\|\Big(\sum_{n\in \NN_0}\big|
P_{2^n}S_{j+n}f\big|^2\Big)^{1/2}\bigg\|_{\ell^p}\lesssim\|f\|_{\ell^p},
\end{equation}
and
\begin{equation}
\label{eq:33} 
\bigg\|\Big(\sum_{n\in \NN_0}\big|  (\mathcal M_{2^n} - 
P_{2^n})S_{j+n}f\big|^2\Big)^{1/2}\bigg\|_{\ell^2}\lesssim2^{-|j|/2}\|f\|_{\ell^2}
\end{equation}
then interpolation does the job. 
To prove \eqref{eq:32} we first show  the following dimension-free vector-valued bounds
\begin{align}
\label{eq:34}
  \bigg\|\Big(\sum_{n\in \NN_0}\big|\mathcal M_{2^n} g_n\big|^2\Big)^{1/2}\bigg\|_{\ell^p}\lesssim
  \bigg\|\Big(\sum_{n\in \NN_0}|g_n|^2\Big)^{1/2}\bigg\|_{\ell^p}
\end{align}
and
\begin{align}
\label{eq:35}
  \bigg\|\Big(\sum_{n\in \NN_0}\big|P_{2^n} g_n\big|^2\Big)^{1/2}\bigg\|_{\ell^p}\lesssim
  \bigg\|\Big(\sum_{n\in \NN_0}|g_n|^2\Big)^{1/2}\bigg\|_{\ell^p}
\end{align}
for all $p\in(1, \infty)$. Then in view of \eqref{eq:98} we conclude
\begin{align*}
\bigg\|\Big(\sum_{n\in \NN_0}\big|  \mathcal  M_{2^n} S_{j+n}f \big|^2\Big)^{1/2}\bigg\|_{\ell^p}+\bigg\|\Big(\sum_{n\in \NN_0}\big|
P_{2^n}S_{j+n}f\big|^2\Big)^{1/2}\bigg\|_{\ell^p}\lesssim\bigg\|\Big(\sum_{n\in \NN_0}\big|
S_{j+n}f\big|^2\Big)^{1/2}\bigg\|_{\ell^p}\lesssim\|f\|_{\ell^p},
\end{align*}
which proves \eqref{eq:32}.

The proof of \eqref{eq:34} and
\eqref{eq:35} follows respectively from \eqref{eq:87} (with $a_n=2^n$) and
\eqref{eq:95} and a vector-valued interpolation.  We only demonstrate
\eqref{eq:34}, the estimate in \eqref{eq:35} will be obtained
similarly. Indeed, for $p\in(1, \infty)$ and
$s\in[1, \infty]$, let $A(p, s)$ be the best constant in the following
inequality
\begin{align*}
  \bigg\|\Big(\sum_{n\in \NN_0}\big|\mathcal M_{2^n} g_n\big|^s\Big)^{1/s}\bigg\|_{\ell^p}\le A(p, s)
  \bigg\|\Big(\sum_{n\in \NN_0}|g_n|^s\Big)^{1/s}\bigg\|_{\ell^p}.
\end{align*}
Then interpolation, duality ($A(p, s)=A(p', s')$), and \eqref{eq:87}  yield \eqref{eq:34}, since
\[
A(p, 2)\le A(p, 1)^{1/2}A(p, \infty)^{1/2}=A(p', \infty)^{1/2}A(p, \infty)^{1/2}\le C_{p', \infty}^{1/2}C_{p, \infty}^{1/2}.
\]

By Plancherel's theorem to prove \eqref{eq:33} we need to estimate
\begin{align*}
&\bigg\|\Big(\sum_{n\in \NN_0}\big|  (\mathcal  M_{2^n} - 
P_{2^n})S_{j+n}f\big|^2\Big)^{1/2}\bigg\|_{L^2}\\
  &=\bigg(\int_{\TT^d}\sum_{n\in
    \NN_0}\big|\mathfrak n_{2^n}(\xi)\big(e^{-2^{n+j} |\xi|_{\st}}
  -e^{-2^{n+j-1} |\xi|_{\st}}\big)\big|^2
  |\hat{f}(\xi)|^2{\rm d}\xi\bigg)^{1/2},
\end{align*}
with $\mathfrak n_{2^n}$ defined in \eqref{eq:kndef}. This has been already done in \eqref{eq:99+} (for $a_n=2^n$ and $a=2$) and thus \eqref{eq:33} holds. This completes the proof of \eqref{eq:31} and hence also the proof of Theorem \ref{thm:11}.

\subsection{Short variations: proof of Theorem \ref{thm:10}}
In view of inequalities \eqref{eq:18-} (with $Z_1=(0,1)$ and $Z_2=[1,\infty)$) and \eqref{eq:29} (with $Z=[1,\infty)$), and Theorem \ref{thm:1} and Theorem \ref{thm:11}, the proof
of Theorem \ref{thm:10} will be completed if we show that for every
$p\in(3/2, 4)$ there is a constant $C_p>0$ such that for every
$f\in\ell^p(\ZZ^d)$ we have
\begin{align}
\label{eq:36}
     \bigg\|\Big(\sum_{n\in\NN_0}V_2\big(\mathcal M_{t}f: t\in[2^n,
   2^{n+1})\big)^2\Big)^{1/2}\bigg\|_{\ell^p}\le C_{p}\|f\|_{\ell^p}.
\end{align}
As in \cite{BMSW1} the essential tool will be inequality \eqref{eq:37}
from Lemma \ref{lem:5}.
\begin{lemma}
\label{lem:5}
For every  $n\in\NN_{0}$, for every $r\ge1$ and for   every function $\mathfrak a:[2^n,
        2^{n+1}]\to\CC$ satisfying $\mathfrak a(t)=\mathfrak a(\lfloor t\rfloor)$  we have
        \begin{align}
\label{eq:37}
          \begin{split}
        	V_r\big(\mathfrak a_t: t\in[2^n, 2^{n+1})\big)
		\le 2^{1-1/r}
		\sum_{0\le l\le n}
		\Big(
		\sum_{k = 0}^{2^{l}-1}
		\big|\mathfrak a_{2^n+{2^{n-l}(k+1)}} - \mathfrak a_{2^n+{2^{n-l}k}}\big|^r
		\Big)^{1/r}.
          \end{split}          
        \end{align}
\end{lemma}
We refer to \cite[Lemma 2.1]{BMSW1} for the proof. The
advantage of this inequality is that the variational seminorm on a
dyadic block is controlled by a sum of suitable square functions, which are better adjusted to investigations on $\ell^p(\ZZ^d)$ spaces.
In order to prove \eqref{eq:36} it suffices to show that there are $\delta_p, \varepsilon_p\in(0, 1)$ such that for every $l\in\NN$, for every $j\in\ZZ$ and for every
$f\in\ell^p(\ZZ^d)$ we have 
\begin{align}
  \label{eq:38}
\bigg\|\Big(\sum_{n\ge l} \sum_{k = 0}^{2^{l}-1} \big|(\mathcal M_{2^n+{2^{n-l}(k+1)}}
  - \mathcal M_{2^n+{2^{n-l}k}})S_{j+n}f\big|^2
  \Big)^{1/2}\bigg\|_{\ell^p}
  \lesssim 2^{-\delta_pl}2^{-\varepsilon_p|j|}\|f\|_{\ell^p}.    
\end{align}
Once \eqref{eq:38} is established we appeal to \eqref{eq:97}, \eqref{eq:37} and \eqref{eq:38} and  obtain
\begin{align*}
       \bigg\|\Big(\sum_{n\in\NN_0}V_2\big(\mathcal M_{t}f: t\in[2^n,
  2^{n+1})\big)^2\Big)^{1/2}\bigg\|_{\ell^p}\lesssim
\sum_{j\in\ZZ}\sum_{l\ge0}2^{-\delta_pl}2^{-\varepsilon_p|j|}\|f\|_{\ell^p}  \lesssim\|f\|_{\ell^p}.
\end{align*}
The  proof of inequality \eqref{eq:38} will be given in the next two paragraphs.
\subsubsection{Proof of inequality \eqref{eq:38} for $p\in(3/2, 2]$} 
The  proof is based on ideas from the proof of Proposition \ref{prop:2}. Throughout the proof we fix $l\in \NN.$ 
Take $N\in\NN$ and for $s\in(3/2, 2]$ and $r\in[1, \infty]$ let  $A_N(s, r)$ be the smallest constant in the following inequality
    \begin{align}
\label{eq:39}
      \bigg\|\Big(\sum_{l\le n\le N} \sum_{k = 0}^{2^{l}-1} \big|(\mathcal M_{2^n+{2^{n-l}(k+1)}}
  - \mathcal M_{2^n+{2^{n-l}k}})g_n\big|^r
  \Big)^{1/r}\bigg\|_{\ell^s}\le A_N(s, r)\bigg\|\Big(\sum_{l\le n\le N}|g_n
    |^r\Big)^{1/r}\bigg\|_{\ell^s}.
    \end{align}
   By Minkowski's inequality, since $\|\mathcal M_{t}f\|_{\ell^s}\le\|f\|_{\ell^s}$ for all $s\in[1, \infty]$,  we have
    \begin{align*}
   \bigg\|\Big(\sum_{l\le n\le N} \sum_{k = 0}^{2^{l}-1} \big|(\mathcal M_{2^n+{2^{n-l}(k+1)}}
   - \mathcal M_{2^n+{2^{n-l}k}})g_n\big|^r
   \Big)^{1/r}\bigg\|_{\ell^s}\le  N 2^{l+1} \bigg\|\Big(\sum_{l\le n\le N}|g_n
   |^r\Big)^{1/r}\bigg\|_{\ell^s},
   \end{align*}
   hence, it follows that $A_N(s, r)\le N 2^{l+1}<\infty$. Let $u\in(1, p)$ be such that
      \[
\frac{1}{u}=\frac{1}{2}+\frac{1}{2p}.
\]
Now it is not difficult to see that $A_N(1, 1)\le 2^{l+1}$. Indeed, using $\|\mathcal M_{t}f\|_{\ell^1}\le\|f\|_{\ell^1}$ we obtain
\begin{align*}
\bigg\|\sum_{l\le n\le N} \sum_{k = 0}^{2^{l}-1} \big|(\mathcal M_{2^n+{2^{n-l}(k+1)}}
- \mathcal M_{2^n+{2^{n-l}k}})g_n\big|
\bigg\|_{\ell^1} \le 2^{l+1} \bigg\|\sum_{l\le n\le N} |g_n| \bigg\|_{\ell^1}.
\end{align*}
Moreover, by Theorem \ref{thm:1}, if 
$g=\sup_{l\le n\le N}|g_n|$ then
\begin{align*}
  \big\|\sup_{l\le n\le N}\sup_{0\le k<2^l}\big|(\mathcal M_{2^n+{2^{n-l}(k+1)}}
  - \mathcal M_{2^n+{2^{n-l}k}})g_n
    \big|\big\|_{\ell^p}\le 2\big\|\sup_{t>0}\mathcal M_{t}g\big\|_{\ell^p}\le 2C_p\|g\|_{\ell^p}.
\end{align*}
Hence, using complex  interpolation method, see \cite[Theorem 5.1.2]{BeLo}, we obtain
$$A_N(u, 2)\le A_N(1, 1)^{1/2}A_N(p, \infty)^{1/2}\lesssim 2^{l/2}.$$
Then by \eqref{eq:39} and \eqref{eq:98} (with $a_n=2^n$) we get
\begin{align}
  \label{eq:20}
  \begin{split}
      \bigg\|\Big(\sum_{l\le n\le N} &\sum_{k = 0}^{2^{l}-1} \big|(\mathcal M_{2^n+{2^{n-l}(k+1)}}
  - \mathcal M_{2^n+{2^{n-l}k}})S_{j+n}f\big|^2
  \Big)^{1/2}\bigg\|_{\ell^u}\\
  &\le A_N(u, 2)\bigg\|\Big(\sum_{n\in\ZZ}|S_{j+n}f
    |^2\Big)^{1/2}\bigg\|_{\ell^u}\lesssim 2^{l/2}\|f\|_{\ell^u}.
  \end{split}
\end{align}
We now take  $\rho\in(0, 1]$ satisfying
    \[
\frac{1}{p}=\frac{1-\rho}{u}+\frac{\rho}{2},
\]
then $\rho=p-1$ and  $1-\rho=2-p$. Interpolation between \eqref{eq:20} and \eqref{eq:105} (with $0<\varepsilon< 2-1/(p-1)$) yields
\begin{align*}
  \begin{split}
      \bigg\|\Big(\sum_{l\le n\le N} \sum_{k = 0}^{2^{l}-1}& \big|(\mathcal M_{2^n+{2^{n-l}(k+1)}}
  - \mathcal M_{2^n+{2^{n-l}k}})S_{j+n}f\big|^2
  \Big)^{1/2}\bigg\|_{\ell^p}\\
&\lesssim 2^{l(1-\rho)/2}2^{-(1-\varepsilon)\rho l/2}2^{-\varepsilon|j|\rho/4}\|f\|_{\ell^p}\\
&\lesssim 2^{l(2-p)/2}2^{-(1-\varepsilon)(p-1)l/2}2^{-\varepsilon|j|(p-1)/4}\|f\|_{\ell^p}\\
&  \lesssim 2^{-\delta_p l}2^{-\varepsilon_p|j|}\|f\|_{\ell^p},
  \end{split}
\end{align*}
where $\delta_p=\frac{(1-\varepsilon)(p-1)}{2}-\frac{2-p}{2}>0$, if $p\in(3/2, 2]$ and $\varepsilon_p=\frac{\varepsilon(p-1)}{4}$.  This completes the proof.
\subsubsection{Proof of inequality \eqref{eq:38} for $p\in(2, 4)$}

To this end, we show that, for $p\in [2,\infty)$, we have
\begin{align}
\label{eq:45}
\bigg\|\Big(\sum_{n\ge l} \sum_{k = 0}^{2^{l}-1} \big|(\mathcal M_{2^n+{2^{n-l}(k+1)}}
  - \mathcal M_{2^n+{2^{n-l}k}})S_{j+n}f\big|^2
  \Big)^{1/2}\bigg\|_{\ell^p}
  \lesssim 2^{l/2}\|f\|_{\ell^p}.  
            \end{align}
 Then interpolation of \eqref{eq:45} with \eqref{eq:105}   does the job and we
 obtain \eqref{eq:93} for all $p\in[2, 4).$           
            
           Thus we focus on proving \eqref{eq:45}. Since $p\ge 2$ we estimate
            \begin{align*}
\bigg\|\Big(\sum_{n\ge l} \sum_{k = 0}^{2^{l}-1} &\big|(\mathcal M_{2^n+{2^{n-l}(k+1)}}
  - \mathcal M_{2^n+{2^{n-l}k}})S_{j+n}f\big|^2
  \Big)^{1/2}\bigg\|_{\ell^p}^2\\
&  \le 2^l\max_{0\le k<2^l}\bigg\|\sum_{n\ge l}  \big|(\mathcal M_{2^n+{2^{n-l}(k+1)}}
  - \mathcal M_{2^n+{2^{n-l}k}})S_{j+n}f\big|^2
  \bigg\|_{\ell^{p/2}}\\
 & \lesssim 2^l\max_{0\le k\le 2^l}\bigg\|\Big(\sum_{n\ge l}  \big|\mathcal M_{2^n+{2^{n-l}k}}S_{j+n}f\big|^2
   \Big)^{1/2}\bigg\|_{\ell^{p}}^2\\
   &\lesssim 2^l\|f\|_{\ell^{p}}^2,
            \end{align*}
            where the last inequality follows from \eqref{eq:98} (with $a_n=2^n$) and 
            \begin{align}
             \label{eq:46}
              \sup_{l\ge0}\max_{0\le k\le 2^l}\bigg\|\Big(\sum_{n\ge l}  \big|\mathcal M_{2^n+{2^{n-l}k}}g_n\big|^2
  \Big)^{1/2}\bigg\|_{\ell^{p}}\lesssim\bigg\|\Big(\sum_{n\in\ZZ}  |g_n|^2
  \Big)^{1/2}\bigg\|_{\ell^{p}},
            \end{align}
            which holds for all $p\in(1, \infty)$ and the implicit
            constant independent of the dimension. To prove
            \eqref{eq:46} we follow the argument used to justify
            \eqref{eq:34}.  This is feasible, since for every
            $p\in(1, \infty]$ and for
            every $f\in \ell^p(\ZZ^d)$ we have the following lacunary  estimate
            \[
\sup_{l\ge0}\max_{0\le k\le2^l}\big\|\sup_{n\ge l}|\mathcal M_{2^n+{2^{n-l}k}}f|\big\|_{\ell^p}\le C_{p,\infty}\|f\|_{\ell^p}.
\]
The above is a consequence of Theorem \ref{thm:0'} with the lacunary sequence $a_{n}=(1+2^{-l}k)2^n$ and $a=2$.  

\subsection{Transference principle to the ergodic setting: proof of Theorem \ref{thm:12}}
\label{sec:tp} Recall that
$(X, \mathcal B(X), \mu)$ is a $\sigma$-finite measure space with a
family of commuting and invertible measure-preserving transformations
$T_1,\ldots, T_d$, which map $X$ to itself.
In Proposition \ref{prop:3} we prove the transference principle, which will allow
us to deduce estimates for $r$-variations on $L^p(X)$ for the
operator $\mathcal A_t^G$ defined in \eqref{eq:118} from the corresponding bounds for $\mathcal M_t^G$ on
$\ell^p(\ZZ^d)$. Theorem \ref{thm:12} will follow directly from Proposition \ref{prop:3} combined with Theorem \ref{thm:10} and Theorem \ref{thm:11}. 
\begin{proposition}
\label{prop:3}
Given $p\in(1, \infty)$ and $r\in(2, \infty]$ suppose that there is
a constant $C_{p, r}>0$ such that for a
symmetric convex body $G\subset\RR^d$ the following estimate
\begin{align}
  \label{eq:93}
  \big\|V_r\big(  \mathcal M_t^G f: t\in Z)\big)\big\|_{\ell^p(\ZZ^d)}\le C_{p, r}\|f\|_{\ell^p(\ZZ^d)}
\end{align}
holds for all $f\in \ell^p(\ZZ^d)$ with the implied constant
independent of $d\in \NN$, where $Z\subseteq (0, \infty)$.   Let $\mathcal A_t^G$ be
the ergodic counterpart of $\mathcal M_t^G$.
Then for every $h\in L
^p(X)$ the inequality
\begin{align}
  \label{eq:110}
  \big\|V_r\big(  \mathcal A_t^G h: t\in Z)\big)\big\|_{L^p(X)}\le C_{p,r}\,\|h\|_{L^p(X)}
\end{align}
holds  with the parameters $p$, $r,$ and the constant $C_{p,r}$ as in \eqref{eq:93}.

\end{proposition}

\begin{proof}
For any convex symmetric body $G\subseteq\RR^d$ there is
a constant $c=c_G\in\NN$ such that $G_t\subseteq
Q_{ct}$ for
every $t>0$.  We fix $f\in L^p(X)$, $\varepsilon>0,$ and $R\in\NN$.
Let us define for every $x\in X$ the function
\[
\phi_x(y)=
\begin{cases}
f\big(T_1^{y_1}\circ\ldots\circ T_d^{y_d}x\big), & \text{ if } |y|_{\infty}\le cR(1+\varepsilon/d),\\
0, & \text{ otherwise}.
\end{cases}
\]
Observe that  
for every $z\in Q_{cR}$ and $t<R\varepsilon/d,$ we have
\begin{align}
\label{eq:111}
  \begin{split}
\mathcal A_t^Gf\big(T_1^{z_1}\circ\ldots\circ T_d^{z_d}x\big)&=
\frac{1}{|G_t\cap\ZZ^d|}\sum_{y\in G_t\cap\ZZ^d}f\big(T_1^{z_1-y_1}\circ\ldots\circ T_d^{z_d-y_d    }x\big)\\
&=\frac{1}{|G_t\cap\ZZ^d|}\sum_{y\in G_t\cap\ZZ^d}\phi_x(z-y)=\mathcal M_t^{G}\phi_x(z),
  \end{split}
\end{align}
since $|z-y|_{\infty}\le cR(1+\varepsilon/d)$, whenever $z\in Q_{cR}$ and $y\in G_t$.
 Hence, by \eqref{eq:93} and \eqref{eq:111} we get
\begin{align}
  \label{eq:108}
  \begin{split}
  \sum_{z\in Q_{cR}\cap\ZZ^d}\big|V_r\big(\mathcal A_t^Gf\big(T_1^{z_1}\circ&\ldots\circ T_d^{z_d}x\big): t\in Z\cap(0, R\varepsilon/d)\big)\big|^p\\
  &  \le     \sum_{z\in Q_{cR}\cap\ZZ^d}\big|V_r\big(\mathcal M_t^G\phi_x(z): t\in
  Z\cap(0, R\varepsilon/d)\big)\big|^p\\
  &\le \big\|V_r\big(\mathcal M_t^G\phi_x: t\in Z\big)\big\|_{\ell^p(\ZZ^d)}^p\\
  &\le C_{p, r}^p\|\phi_x\|_{\ell^p(\ZZ^d)}^p.
  \end{split}
\end{align}
Averaging \eqref{eq:108} over $x\in X$ we obtain
\begin{align}
  \label{eq:109}
  \begin{split}
  \sum_{z\in Q_{cR}\cap\ZZ^d}\big\|V_r\big(\mathcal A_t^Gf\big(T_1^{z_1}\circ&\ldots\circ T_d^{z_d}x\big): t\in Z\cap(0, R\varepsilon/d)\big)\big\|_{L^p(X)}^p\\
  &\le  C_{p, r}^p\sum_{z\in Q_{cR(1+\varepsilon/d)}\cap\ZZ^d}\big\|f\big(T_1^{z_1}\circ\ldots\circ T_d^{z_d}x\big)\big\|_{L^p(X)}^p,
  \end{split}
\end{align}
by definition of $\phi_x$. Inequality \eqref{eq:109} guarantees that
\begin{align*}
  (2cR)^d\cdot \big\|V_r\big(\mathcal A_t^Gf: t\in Z\cap(0, R\varepsilon/d)\big)\big\|_{L^p(X)}^p
  \le C_{p, r}^p\cdot(2cR(1+\varepsilon/d)+1)^d\cdot \|f\|_{L^p(X)}^p,
\end{align*}
since all $T_1^{z_1},\ldots,T_d^{z_d}$ preserve the measure $\mu$ on $X$. Dividing both sides by $(2cR)^d$ we obtain
that
\begin{equation*}
\begin{split}
  \big\|V_r\big(\mathcal A_t^Gf: t\in Z\cap(0, R\varepsilon/d)\big)\big\|_{L^p(X)}^p
  \le  C_{p, r}^p\bigg(\bigg(1+\frac{\varepsilon}{d}\bigg)+\frac{1}{2cR}\bigg)^d\|f\|_{L^p(X)}^p.
  \end{split}
  \end{equation*}
  Taking $R\to\infty$ and invoking the monotone convergence theorem we conclude that
  \begin{align*}
    \big\|V_r\big(\mathcal A_t^Gf: t\in Z\big)\big\|_{L^p(X)}^p
  \le  C_{p, r}^p\bigg(1+\frac{\varepsilon}{d}\bigg)^d\|f\|_{L^p(X)}^p\le C_{p, r}^p\:e^{\varepsilon}\|f\|_{L^p(X)}^p.
  \end{align*}
  Finally, letting $\varepsilon\to 0^+$ we obtain \eqref{eq:110} and complete the proof of the proposition.
\end{proof}

\appendix

\section{Lust--Piquard's counterexample for the Riesz transforms}
\label{sec:6}
For every $j\in\{1,\ldots,d\}$ let $\Delta_j$, and $\mathcal L$ be defined as
in Section \ref{sec:4}.  For  every $x\in\ZZ^d$ and $j\in\{1,\ldots,d\}$ we consider the discrete $j$-th Riesz transform given by
\[
\mathcal R_j f(x):=\frac12 \Delta_j \mathcal L^{-1/2}f(x).
  \]

In \cite{Lu_Piqu1} Lust--Piquard  proved the following theorem.
\begin{theorem}
	\label{thm:LuPi}
For every $p\in[2,\infty)$ there is $C_p>0$ independent of $d\in\NN$  such
that for every $f\in \ell^p(\ZZ^d)$ we have
\begin{align}
  \label{eq:51}
\bigg\|\Big(\sum_{j=1}^d |\mathcal R_jf|^2+|\mathcal R_j^* f|^2\Big)^{1/2}\bigg\|_{\ell^p}\le C_p\|f\|_{\ell^p}.  
\end{align}
\end{theorem}
\noindent  It was also proved in \cite{Lu_Piqu1} that the inequality \eqref{eq:51} for $p\in(1, 2)$ involves the bound which depends on the dimension. Our next result quantifies this dependence. 
\begin{proposition}
\label{lem:corLuPi}
	For every $q\in(1, 2)$ and $\varepsilon >0$ there is $C_{q,\varepsilon}>0$ independent of $d\in\NN$ and such that for every $f\in \ell^q(\ZZ^d)$ we have
	\begin{equation}
	\label{eq:5}
	\bigg\|\Big(\sum_{j=1}^d |\mathcal R_jf|^2+|\mathcal R_j^*
        f|^2\Big)^{1/2}\bigg\|_{\ell^q}
        \le C_{q,\varepsilon}d^{1/q-1/2+\varepsilon}\|f\|_{\ell^q}.
        \end{equation}
        The bound in \eqref{eq:5} is essentially sharp. Namely, there exists $C_q>0$ such that for all $d\in \NN$ we have
		\begin{equation}
		\label{eq:6a}
		\sup_{0<\|f\|_{\ell^q}\le1}\bigg\|\Big(\sum_{j=1}^d
                |\mathcal R_jf|^2\Big)^{1/2}\bigg\|_{\ell^q} \|f\|_{\ell^q}^{-1}\ge C_qd^{1/q-1/2}.
                \end{equation}
\end{proposition}
\begin{proof}
We first demonstrate \eqref{eq:5}. By Khintchine's inequality and Fubini's theorem we have
\begin{equation*}
	\bigg\|\Big(\sum_{j=1}^d |\mathcal R_jf|^2+|\mathcal R_j^*
        f|^2\Big)^{1/2}\bigg\|_{\ell^q}^q \simeq_q \EE
        \bigg\|\sum_{j=1}^d \eps_j \mathcal R_jf \bigg\|_{\ell^q}^q+ \EE
        \bigg\|\sum_{j=1}^d \eps_j \mathcal R_j^*f \bigg\|_{\ell^q}^q,
\end{equation*}
where $\eps_j\in\{-1,1\}$ are independent and identically distributed
Rademacher variables.  We  note that
\begin{align*}
  \Big\|\sum_{j=1}^d \eps_j\mathcal R_jf
  \Big\|_{\ell^q}
  &=\sup_{\|g\|_{\ell^{q'}}\le 1}\Big|\sum_{x\in\ZZ^d}f(x)\sum_{j=1}^d \overline{\eps_j\mathcal R_j^*g(x)}\Big| && \text{by duality}\\
  &\le\|f\|_{\ell^q}\sup_{\|g\|_{\ell^{q'}}\le 1}\Big\|\sum_{j=1}^d \eps_j\mathcal R_j^*g
    \Big\|_{\ell^{q'}}&&\text{by H\"older's inequality}\\
  &\le d^{1/2}\|f\|_{\ell^q}\sup_{\|g\|_{\ell^{q'}}\le 1}\bigg\|\Big(\sum_{j=1}^d |\mathcal
    R_j^*g|^2\Big)^{1/2}\bigg\|_{\ell^{q'}}&&\text{by Cauchy--Schwarz inequality}\\
  &\le C_{q'}d^{1/2}\|f\|_{\ell^q}&&\text{by Theorem \ref{thm:LuPi}}.
\end{align*}
The same inequality holds with $\mathcal R_j^*$ in place of $\mathcal R_j$ and we conclude that
\[
\bigg\|\Big(\sum_{j=1}^d |\mathcal R_jf|^2+|\mathcal R_j^* f|^2\Big)^{1/2}\bigg\|_{\ell^q}\le C_{q'}d^{1/2}\|f\|_{\ell^q}.
\]
Interpolating the last bound for $q>1$ (which is close to $1$) with \eqref{eq:51} (for $p=2$) we obtain \eqref{eq:5} with $\varepsilon$ loss for arbitrary small $\varepsilon>0$.

We now demonstrate \eqref{eq:6a}. Here we follow \cite[Proposition 2.9]{Lu_Piqu1} but we are keen on keeping the dependence on $q$ and $d.$ Let $g=\delta_0$ be the Dirac delta at zero in $\ZZ $
and consider
\[
G(x)=\prod_{k=1}^d g(x_k) \quad \text{for} \quad x\in \ZZ^d.
\]
Then
\[
\|g\|_{\ell^q(\ZZ)}=\|G\|_{\ell^q(\ZZ^d)}=1.
\]
Let $\Delta$ denote the discrete derivative on $\ZZ$, i.e. 
\[
\Delta g(y)=g(y)-g(y+1)\quad \text{for} \quad y\in \ZZ.
\]
Then for every $j\in\{1, \ldots, d\}$, with $\Delta_j$ as in Section \ref{sec:4}, we have
\[
\Delta_j G(x)=\Delta g(x_j)\prod_{k\neq j}g(x_k).
\]
For $j\in\{1, \ldots, d\}$ we define 
$$E_j=\{0\}\times\ldots \times \{0\}^c\times \ldots\times \{0\},$$
where $\{0\}^c$ occurs in the $j$-th factor. Then the sets $E_j$ are disjoint. We note that
\begin{align}
 \nonumber \bigg\|\Big(\sum_{j=1}^d |\Delta_j G|^2\Big)^{1/2}\bigg\|_{\ell^q(\ZZ^d)}
   &\ge \bigg\|\Big(\sum_{j=1}^d \ind{E_j}\Big)^{1/2}\Big(\sum_{j=1}^d |\Delta_j G|^2\Big)^{1/2}\bigg\|_{\ell^q(\ZZ^d)} && \text{since $\ind{\ZZ^d}\ge\sum_{j=1}^d \ind{E_j}$} \\
\nonumber  &\ge  \Big\|\sum_{j=1}^d \ind{E_j} |\Delta_j G|\Big\|_{\ell^q(\ZZ^d)}&&\text{by Cauchy--Schwarz inequality}\\
         \nonumber                  &=\bigg(\sum_{j=1}^d\,\Big\|\big(\ind{\{0\}^c}\Delta g(x_j)\big)\prod_{k\neq j}g(x_k)\Big\|_{\ell^q(\ZZ^d)}^q\bigg)^{1/q}
  &&\text{by disjointness of $E_j$'s}\\
  &=d^{1/q}\|\ind{\{0\}^c}\Delta g\|_{\ell^q(\ZZ)}.\label{eq:8}
\end{align}

We will use the following inequality
\begin{align}
  \label{eq:53}
  \|\mathcal L^{1/2}G\|_{\ell^q(\ZZ^d)}\le 2 \|\mathcal LG\|_{\ell^q(\ZZ^d)}^{1/2} \|G\|_{\ell^q(\ZZ^d)}^{1/2}.
\end{align}
Indeed, by the Taylor formula with integral reminder we have
\[
e^{-t\mathcal L^{1/2}}=I-t\mathcal
L^{1/2}+\int_0^{t}(t-u)\,e^{-u\mathcal L^{1/2}}\mathcal L\dif u \quad \text{for} \quad t>0.
\]
This implies that
\[
\|\mathcal L^{1/2}G\|_{\ell^q(\ZZ^d)}\leq t^{-1}\|G-e^{-t\mathcal
  L^{1/2}}G\|_{\ell^q(\ZZ^d)}+t^{-1}\int_0^{t}(t-u)\,\|e^{-u\mathcal
  L^{1/2}}\mathcal LG\|_{\ell^q(\ZZ^d)}\dif u,
\]
which together with the contractivity of $e^{-u\mathcal L^{1/2}}$ on $\ell^q(\ZZ^d)$ (see \eqref{eq:contrLa}) gives
$$\|\mathcal L^{1/2}G\|_{\ell^q(\ZZ^d)}\leq \frac2t\|G\|_{\ell^q(\ZZ^d)}+\frac
t2 \|\mathcal LG\|_{\ell^q(\ZZ^d)}.$$
Optimizing over $t>0$ we obtain  \eqref{eq:53}. 

We now observe 
\[
\mathcal L G(x)=\frac14\sum_{j=1}^d \Delta_j \Delta_j^* G(x)=\frac14 \sum_{j=1}^d \Delta\Delta^*g(x_j)\prod_{k\neq j}g(x_k)
\]
and consequently obtain
\[
\|\mathcal LG\|_{\ell^q(\ZZ^d)}\leq \frac d{4}\|\Delta \Delta^* g\|_{\ell^q(\ZZ)},
\]
which combined with \eqref{eq:53}  implies
\begin{equation}
\label{eq:10}
\|\mathcal L^{1/2}G\|_{\ell^q(\ZZ^d)}\le 2 \|\mathcal LG\|_{\ell^q(\ZZ^d)}^{1/2} \|G\|_{\ell^q(\ZZ^d)}^{1/2} \le  d^{1/2}\|\Delta \Delta^* g\|_{\ell^q(\ZZ)}^{1/2}.\end{equation}

Combining \eqref{eq:10} with \eqref{eq:8} we see that
\begin{equation}
\label{eq:9a}
\bigg\|\Big(\sum_{j=1}^d |\Delta_j
G|^2\Big)^{1/2}\bigg\|_{\ell^q(\ZZ^d)}\|\mathcal L^{1/2}G\|_{\ell^q(\ZZ^d)}^{-1}\ge B d^{1/q-1/2},
\end{equation}
where
\[
B:=\frac{\|\ind{\{0\}^c}\Delta g\|_{\ell^q(\ZZ)}}{ \|\Delta
  \Delta^* g\|_{\ell^q(\ZZ)}^{1/2}}.
\]
Since
$\Delta g(-1)=-1$ we see that $B\neq 0$. We remark that the non-local nature of the derivative $\Delta$ plays an essential role here.

To complete the proof of \eqref{eq:6a} we assume for a contradiction that for all $C_q>0$ there is 
$d\in \NN$ such that for all $f\in \ell^q(\ZZ^d)$ we have
\begin{align*}
  \bigg\|\Big(\sum_{j=1}^d |\mathcal R_jf|^2\Big)^{1/2}\bigg\|_{\ell^q(\ZZ^d)}\|f\|_{\ell^q(\ZZ^d)}^{-1}
  \le C_q\, d^{1/q-1/2}.
\end{align*}
But this contradicts \eqref{eq:9a} 
by taking $C_q=B/4$ and $f=\mathcal L^{1/2}G$, since $\mathcal R_j f=\frac12\Delta_j G$. \end{proof}


\begin{thebibliography}{16}
\bibitem{Ald1} J.M. Aldaz, \textit{The weak type (1, 1) bounds for the
  maximal function associated to cubes grow to infinity with the
  dimension}, Ann.  Math.  {\bf 173}, no. 2, (2011), pp. 1013--1023.
	

\bibitem{BeLo} J. Bergh, J. L\"ofstr\"om \textit{Interpolation Spaces - An Introduction}, Springer-Verlag, Berlin - Heidelberg - New York, 1976.

\bibitem{B1} J. Bourgain, \textit{On high dimensional maximal
  functions associated to convex bodies}, Amer. J.  Math. {\bf 108},
(1986), pp. 1467--1476.
        
\bibitem{B2} J. Bourgain, \textit{On $L^p$ bounds for maximal
  functions associated to convex bodies in $\RR^n$}, Israel J. Math.
Math. {\bf 54} (1986), pp. 257--265.

\bibitem{B3} J. Bourgain, \textit{On the Hardy-Littlewood maximal
  function for the cube}, Israel J. Math.  {\bf 203} (2014),
pp. 275--293.

\bibitem{BMSW1} J. Bourgain, M. Mirek, E. Stein, B. Wr\'obel,
\textit{Dimension-free variational estimates on $L^p(\RR^d)$ for
  symmetric convex bodies}, Geom. Funct. Anal.
  {\bf 28}, no. 1, (2018),  pp. 58-99.

\bibitem{BMSW2} J. Bourgain, M. Mirek, E. Stein, B. Wr\'obel,
\textit{On discrete Hardy--Littlewood maximal functions 
  over the balls in
  $\ZZ^d$: dimension-free estimates}, arXiv:1812.00154, preprint 2018.

\bibitem{Car1} A. Carbery, \textit{An almost-orthogonality principle
  with applications to maximal functions associated to convex bodies},
Bull. Amer. Math. Soc.  {\bf 14} no. 2 (1986), pp. 269--274.

\bibitem{DGM1} L. Delaval, O. Gu\'edon, B. Maurey,
\textit{Dimension-free bounds for the Hardy-Littlewood maximal
  operator associated to convex sets}, preprint (2016)
https://arxiv.org/abs/1602.02015

\bibitem{JR1} R. L. Jones, K. Reinhold, \textit{Oscillation and
  variation inequalities for convolution powers}, Ergodic Theory and
Dynam. Systems  {\bf 21}, no. 6, (2001), pp. 1809--1829.

\bibitem{JSW} R. L. Jones, A. Seeger, J. Wright, \textit{Strong
  Variational and Jump Inequalities in Harmonic Analysis},
Trans. Amer. Math. Soc.   {\bf 360}, no. 12, (2008), pp. 6711--6742.

\bibitem{LL} A. Lewko, M. Lewko \textit{Estimates for the square
  variation of partial sums of Fourier series and their
  rearrangements}, J. Funct.  Anal. {\bf 262}, no. 6, (2012), pp. 2561--2607.

\bibitem{Lu_Piqu1} F. Lust-Piquard, \textit{Dimension free estimates
  for discrete Riesz transforms on products of abelian groups},
Adv. Math.  {\bf 185}, no. 2, (2004), pp. 289--327.

\bibitem{MTS1} M. Mirek, E. M. Stein, B. Trojan,
\textit{$\ell^p(\ZZ^d)$-estimates for discrete operators of Radon
  type: Variational estimates}, Invent. Math.  {\bf 209}, no. 3 (2017),
pp. 665--748.



\bibitem{MT1} M. Mirek, B. Trojan, \textit{ Discrete maximal functions
  in higher dimensions and applications to ergodic theory}, Amer. J.
Math. {\bf 138}, (2016), no. 6, pp. 1495--1532.

\bibitem{Mul1} D. M\"uller, \textit{A geometric bound for maximal
  functions associated to convex bodies}, Pacific J. Math. {\bf 142},
no. 2, (1990) pp. 297--312.

\bibitem{NaTa} A. Naor, T. Tao, \textit{Random martingales and localization of maximal inequalities}, J. Funct. Anal. {\bf 259} (2010), 731-779.

\bibitem{Ste1} E. M. Stein, \textit{Topics in harmonic analysis
  related to the Littlewood-Paley theory}, Annals of Mathematics
Studies, Princeton University Press 1970, pp. 1-157.

\bibitem{SteinMax} E. M. Stein, \textit{The development of square
  functions in the work of A. Zygmund}, Bull.  Amer. Math. Soc, {\bf
  7}, (1982), pp. 359--376.

\bibitem{SteinRiesz} E. M. Stein, \textit{Some results in harmonic
  analysis in $\mathbb R^n$, $n\to \infty$},
Bull. Amer. Math. Soc. {\bf 9}, (1983), pp. 71--73.

\bibitem{StStr} E. M. Stein, J.O. Str\"omberg, \textit{Behavior of maximal functions in  $\mathbb R^n$ for large $n$},
(1-2) {\bf 21} (1983), Ark. Mat., 259--269
\end{thebibliography}
\end{document}